\documentclass[12pt]{amsart}

\usepackage{amscd} 
\usepackage{amsmath} 
\usepackage{amssymb} 
\usepackage{amsthm}
\usepackage{bigdelim}
\usepackage{color} 
\usepackage{enumerate}
\usepackage{mathrsfs}
\usepackage{multirow}
\usepackage[all]{xy} 
\usepackage[bookmarks=false]{hyperref} 
\newtheorem{thm}{Theorem}[section] 
\newtheorem{cor}[thm]{Corollary}
\newtheorem{prop}[thm]{Proposition}
\newtheorem{conj}[thm]{Conjecture}

\newtheorem{lem}[thm]{Lemma}
\theoremstyle{definition} 
\newtheorem{defn}[thm]{Definition}

\newtheorem{eg}[thm]{Example} 
\theoremstyle{remark}
\newtheorem{rem}[thm]{Remark}

\newtheorem{notation}[thm]{Notation}

\newtheorem*{ack}{Acknowledgements}

\begin{document}
\title[Direct images of pluricanonical bundles]{Direct images of pluricanonical bundles and Frobenius stable canonical rings of fibers}
\author{Sho Ejiri}
\email{shoejiri.math@gmail.com}
\address{Department of Mathematics, Graduate School of Science, Osaka Metropolitan University, Osaka City, Osaka 558--8585, Japan}
\thanks{He was supported by JSPS KAKENHI Grant Number 18J00171.}
\begin{abstract}
In this paper, we study an algebraic fiber space in positive characteristic whose generic fiber $F$ has finitely generated canonical ring and sufficiently large Frobenius stable canonical ring.
An example of such a case is when $F$ is $F$-pure and its dualizing sheaf is invertible and ample.
We treat a Fujita-type conjecture due to Popa and Schnell concerning direct images of pluricanonical bundles, and prove it under some additional hypotheses.
As an application, we show the subadditivity of Kodaira dimensions in some new cases.
We also prove an analog of Fujino's result regarding his Fujita-type conjecture.
\end{abstract}
\maketitle
\setcounter{tocdepth}{1}
\tableofcontents
\section{Introduction} \label{section:intro}
We first recall a conjecture of Popa and Schnell~\cite{PS14} regarding direct images of pluricanonical bundles: 
\begin{conj}[\textup{\cite[Conjecture~1.3]{PS14}}] \label{conj:PS} \samepage
Let $f:X \to Y$ be a morphism of smooth projective varieties over an algebraically closed field of characteristic zero, 
with $Y$ of dimension $n$, 
and let $\mathcal L$ be an ample line bundle on $Y$.
Then, for every $m \ge 1$,  the sheaf 
$$
f_* \omega_X^m \otimes \mathcal L^l
$$
is generated by its global sections for $l \ge m(n+1)$. 
\end{conj}
This is an extension of a famous conjecture due to Fujita~\cite[Conjecture]{Fujita87} to the relative setting. 
Under the additional assumption that $\mathcal L$ is globally generated,
Conjecture~\ref{conj:PS} follows from a result of Koll\'ar~\cite{Kol87} when $m=1$, 
and has been verified by Popa and Schnell~\cite[Theorem~1.4]{PS14} when $m \ge 2$. 
If $f$ is smooth over the complement of a normal crossing divisor on $Y$, 
then we can remove the above additional assumption by using 
a theorem of Kawamata~\cite[Theorem~1.7]{Kaw02} when $m=1$ and $\dim Y\le 4$. 
Deng~\cite{Den20}, Dutta~\cite{Dut20}, Dutta--Murayama~\cite{DM19} and Iwai~\cite{Iwa17} have studied Conjecture~\ref{conj:PS},
and they have given sufficient conditions, in terms of lower bounds on $l$, 
for the sheaf $f_*\omega_X^m \otimes \mathcal L^l$ to be (generically) globally generated. 

Recently, Fujino~\cite{Fuj19} proposed a new generalization of Fujita's freeness conjecture: 
\begin{conj}[\textup{\cite[Conjecture~1.3]{Fuj19}}] \label{conj:Fujino} \samepage
Let $f:X \to Y$ be a surjective morphism of smooth projective varieties over an algebraically closed field of characteristic zero with $\dim Y =n$, 
and let $\mathcal L$ be an ample line bundle on $Y$.
Then, for each $m \ge 1$, the sheaf 
$$
f_* \omega_{X/Y}^m \otimes \omega_Y \otimes \mathcal L^l
$$
is generated by its global sections on $U$ for $l \ge n+1$, 
where $U$ is the largest Zariski open subset of $Y$ such that $f$ is smooth over $U$. 
\end{conj}
Fujino~\cite{Fuj19} has shown that the above sheaf is generically globally generated, 
assuming additionally that either $\mathcal L$ is globally generated or $l \ge n^2 + \mathrm{min}\{2,m\}$. 
 
%
In positive characteristic, however, 
there exists a counterexample to Conjecture~\ref{conj:PS},
even when we add the hypothesis that $\mathcal L$ is globally generated. 
Shentu and Zhang~\cite{SZ20} have observed that 
for a semi-stable fibration $g:S\to \mathbb P^1$ constructed by Moret-Bailly~\cite{MB81}, 
where $S$ is a smooth projective surface,  
we have that $g_* \omega_S \otimes \mathcal O(2)$ 
is isomorphic to $\mathcal O(-1)\oplus\mathcal O(p)$. 
This kind of pathology has been treated more deeply in \cite[Proposition~3.16]{SZ20}. 
 
The purpose of this paper is to prove several generation results 
in positive characteristic. 
We treat morphisms whose generic fibers have sufficiently large Frobenius stable canonical rings.
Recall that the Frobenius stable canonical ring, which was introduced in \cite{HP16a} (and named in \cite{PST17}), 
is a homogeneous subring of the canonical ring of a regular variety $V$, 
whose degree $m$ subgroup is $S^0(V, \omega_V^m)$. 
Recall also that $S^0(V, \omega_V^m)$ is the subspace of $H^0(V,\omega_V^m)$ 
defined by Schwede~\cite{Sch14} as the stable image under the trace maps of the iterations of the Frobenius morphism. 
This notion was generalized to the relative setting by Hacon and Xu~\cite{HX15}. 
 
We fix from now on the following notation. 
Let $k$ be an algebraically closed field of positive characteristic. 
Let $f:X\to Y$ be a surjective morphism of projective varieties over $k$, 
with $X$ smooth and $Y$ of dimension $n$. 
Let $X_\eta$ and $X_{\overline \eta}$ denote the generic fiber and the geometric generic fiber of $f$, respectively. 
We note that $f$ is not necessarily separable. 

Our first main theorem is then stated as follows:
\begin{thm}[\textup{(Theorem~\ref{thm:main}~(2))}] \label{thm:main_intro} \samepage
Suppose that 
\begin{itemize}
\item[(i)]
$\bigoplus_{m \ge 0} H^0(X_{\eta},\omega_{X_{\eta}}^m)$ is 
a finitely generated $k(\eta)$-algebra and that
\item[(i\hspace{-1pt}i)] 
there exists $m_0$ such that 
$
S^0(X_{\eta},\omega_{X_{\eta}}^m)
=H^0(X_{\eta},\omega_{X_{\eta}}^m)
$
for $m \ge m_0$. 
\end{itemize}
Let $\mathcal L$ be a big and globally generated line bundle on $Y$. 
Then the sheaf 
$$
f_* \omega_{X}^m \otimes \mathcal L^l
$$ 
is generically globally generated for $m \ge m_0$ and $l \ge m (n +1)$. 
\end{thm}
One of the features of this theorem is that it does not impose any assumptions on geometric generic fiber $X_{\overline\eta}$ of $f$. 
To the best of the author’s knowledge, all results concerning the positivity of direct images of (relative) pluricanonical bundles impose at least $X_{\overline\eta}$ is reduced (e.g., \cite{Szp79, Kol90, Mau14, Pat14, Eji17}). 
 
The hypotheses of Theorem~\ref{thm:main_intro} hold if $\omega_{X_{\eta}}$ is ample, so we obtain the following:
\begin{thm} \label{thm:ampln_gf} \samepage
Suppose that $\omega_{X_{\eta}}$ is ample. 
Let $\mathcal L$ be an ample and globally generated line bundle on $Y$. 
Then there exists $m_0$ such that the sheaf 
$$
f_* \omega_X^m \otimes \mathcal L^l
$$ 
is generically generated by its global sections for $m \ge m_0$ and $l \ge m(n+1)$. 
\end{thm}
Moreover, in the case when $\omega_X$ is $f$-ample, 
we prove a positive characteristic analog of \cite[Theorem~1.4]{PS14} 
by the same method as that used to prove Theorem~\ref{thm:main_intro}. 
\begin{thm}[\textup{(Theorem~\ref{thm:conclusion}~(2))}] \label{thm:f-ample_intro} \samepage
Suppose that $\omega_X$ is $f$-ample. 
Let $\mathcal L$ be an ample and globally generated line bundle on $Y$. 
Then there exists $m_1$ such that the sheaf 
$$
f_* \omega_X^m \otimes \mathcal L^l
$$ 
is generated by its global sections for $m \ge m_1$ and $l \ge m(n+1)$. 
\end{thm}
The above theorems are proved 
by using two morphisms: the morphism $Y \to \mathbb P^n$ 
defined by a free linear system in $|\mathcal L|$,
and a \textit{separable} endomorphism $\pi$ of $\mathbb P^n$. 
We use this $\pi$ to introduce an invariant measuring 
the positivity of coherent sheaves (Definition~\ref{defn:inv}). 
This is similar to the one introduced in \cite[\S 4]{Eji17}, 
but differs since $\pi$ is separable while 
the Frobenius morphism is employed in \cite{Eji17} instead of $\pi$. 
This difference allows us to investigate $f$ 
without any assumptions on $X_{\overline\eta}$.
 
As long as $Y$ has a generically finite morphism to a variety admitting a special endomorphism (Definition~\ref{defn:epsilon}), the same method performs. 
We apply it to the Albanese morphism of $Y$, and use its consequence to study Iitaka's conjecture.
\begin{thm}
\label{thm:mAd_intro} \samepage
Suppose either that the hypotheses of Theorem~\ref{thm:main_intro} hold or that $\omega_{X_{\eta}}$ is ample. 
In the latter case, let $m_0$ be a sufficiently large integer. 
Further suppose that $Y$ is a smooth projective variety of maximal Albanese dimension.
\begin{itemize}
\item[{\rm (a)}]{\rm (Theorem~\ref{thm:main}~(1))}
Then, for each $m \ge m_0$, the sheaf $f_* \omega_X^m$ is weakly positive in the sense of~\cite{Vie83II}. 
\item[{\rm (b)}] {\rm (Theorem~\ref{thm:Iitaka})}
Assume that $f$ is separable. 
If either $Y$ is a curve or is of general type, then 
$$
\kappa(X) \ge \kappa(Y) +\kappa(X_{\eta}). 
$$
\end{itemize}
\end{thm}
Here, a coherent sheaf $\mathcal F$ on $Y$ is said to be \textit{weakly positive} in the sense of~\cite{Vie83II} if 
for any ample line bundle $\mathcal H$ on $Y$ and any $\alpha \ge 1$, there exists some $\beta \ge 1$ such that 
$(S^{\alpha\beta}(\mathcal F))^{**} \otimes \mathcal H^\beta$ is generically globally generated. 
Note that this definition is weaker than Viehweg's original one~\cite[Definition~1.2]{Vie83}. 

In the case when $Y$ is a curve, Theorem~\ref{thm:mAd_intro}~(b) 
generalizes a result of the author~\cite[Theorem~1.4]{Eji17} 
which needs indeed a stronger assumption than that of Theorem~\ref{thm:mAd_intro}
to prove the weak positivity of sheaves of the form 
$f_* \omega_{X/Y}^m$ (\cite[Theorem~1.1]{Eji17}).
The positivity of these sheaves cannot be obtained from Theorem~\ref{thm:main_intro},
unlike the case of characteristic zero (\cite[Corollary~4.3]{PS14}). 
In fact, there exists a fibration that satisfies hypotheses in Theorem~\ref{thm:main_intro}
but violates the weak positivity theorem (see \cite{Ray78,Xie10,Muk13} or \cite[Example~1.14]{Zha19s}).

We move on to results concerning Conjecture~\ref{conj:Fujino}. 
For the same reason as above, we impose the same conditions as that in~\cite[Theorem~1.1]{Eji17},
which is stronger than that of Theorem~\ref{thm:main_intro}. 
The next theorem can be viewed as a positive characteristic analog of a part of the above result due to Fujino~\cite[Theorem~1.5]{Fuj19}. 
Recall that $X_{\overline\eta}$ denotes the geometric generic fiber of $f$. 
\begin{thm}[\textup{(Theorem~\ref{thm:relative_psef}~(3))}] \label{thm:relative_intro} \samepage
Suppose that $Y$ is smooth. 
Assume that 
\begin{itemize}
\item[(i)]
$\bigoplus_{m \ge 0} H^0(X_{\overline\eta},\omega_{X_{\overline\eta}}^m)$ is 
a finitely generated $k(\overline\eta)$-algebra and that
\item[(i\hspace{-1pt}i')] 
there exists $m_0$ such that 
$
S^0(X_{\overline\eta},\omega_{X_{\overline\eta}}^m)
=H^0(X_{\overline\eta},\omega_{X_{\overline\eta}}^m)
$
for $m \ge m_0$. 
\end{itemize}
Let $\mathcal L$ be a big and globally generated line bundle. 
Then the sheaf
$$
f_* \omega_{X/Y}^m \otimes \omega_Y \otimes \mathcal L^l
$$ 
is generically globally generated for $m \ge m_0$ and $l \ge n+1$. 
\end{thm}
The hypotheses of this theorem hold if 
$X_{\overline\eta}$ 
is $F$-pure (Definition~\ref{defn:F-sing}) 
and has ample dualizing sheaf $\omega_{X_{\overline\eta}}$ (see \cite[\S 3]{Eji17} for more examples), 
so we obtain the following:
\begin{thm}
Suppose that $Y$ is smooth. Assume that $X_{\overline\eta}$ is $F$-pure and $\omega_{X_{\overline\eta}}$ is ample. 
Then there exists $m_0$ such that the sheaf 
$$
f_* \omega_{X/Y}^m \otimes \omega_Y \otimes \mathcal L^l
$$ 
is generically globally generated for $m \ge m_0$ and $l \ge n+1$. 
\end{thm}
In the situation of this theorem, furthermore, we can find an open subset of $Y$ depending only the morphism $f$, 
on which the above sheaves are generated by its global sections. 
\begin{thm}[\textup{(Theorem~\ref{thm:relative_f-ample}~(3))}] \label{thm:flat_intro} \samepage
Suppose that $Y$ is smooth. 
Assume that there exists a dense open subset $Y_0$ of $Y$ such that the following conditions hold:
\begin{itemize}
\item $f$ is flat over $Y_0$;
\item $\omega_X|_{f^{-1}(Y_0)}$ is ample over $Y_0$;
\item every closed fiber of $f$ over $Y_0$ is $F$-pure.
\end{itemize}
Let $\mathcal L$ be an ample and globally generated line bundle on $Y$.
Then there exists $m_0$ such that the sheaf 
$$
f_* \omega_{X/Y}^m \otimes \omega_Y \otimes \mathcal L^l
$$
is generated by its global sections on $Y_0$ for $m \ge m_0$ and $l \ge n+1$. 
\end{thm}

This paper is organized as follows. 
In Section~\ref{section:notation}, we set up notation and terminology. 
Section~\ref{section:trace} deals with the trace maps of iterations of (relative) Frobenius morphisms. 
In Section~\ref{section:positivity}, we generalize and study different notions of base loci of coherent sheaves. 
In Section~\ref{section:invariant}, using a special endomorphism, we introduce an invariant measuring the positivity of coherent sheaves. 
In Section~\ref{section:main}, our main results are stated in a more general setting and proved. 
Section~\ref{section:Iitaka} is devoted to the study of Iitaka's conjecture in positive characteristic. 

\begin{small}
\begin{ack}
The author wishes to express his gratitude to Professors Paolo Cascini, 
Osamu Fujino, Shunsuke Takagi, Hiromu Tanaka and Lei Zhang 
for useful comments and helpful advice.
He is grateful to Professors Yoshinori Gongyo, Zhiyu Tian and Behrouz Taji for valuable comments. 
He would also like to thank Doctors Fabio Bernasconi, Masataka Iwai, Kenta Sato and Shou Yoshikawa 
for stimulating discussions and answering his questions.
Last but not least, he greatly acknowledges many valuable suggestions and helpful comments of the referee.  
He was supported by JSPS KAKENHI Grant Number 18J00171.
\end{ack}
\end{small}
\section{Notation and conventions} \label{section:notation}
Let $k$ be a field. 
By {\it $k$-scheme} we mean a separated scheme of finite type over $k$.  
A {\it variety} is an integral $k$-scheme. 

Let $X$ be an equi-dimensional $k$-scheme of finite type satisfying $S_2$ and $G_1$. 
Here, $S_2$ (resp. $G_1$) stands for Serre's second condition (resp. the condition that it is Gorenstein in codimension~1). 
Let $\mathscr K$ be the sheaf of total quotient rings on $X$. 
An \textit{AC divisor} (or \textit{almost Cartier divisor}) on $X$ 
is a reflexive coherent subsheaf of $\mathscr K$ that 
is invertible on an open subset $U$ of $X$ 
with $\mathrm{codim}(X \setminus U) \ge 2$ (\cite[p. 301]{Har94}, \cite[Definition~2.1]{MS12}). 
Let $D$ be an AC divisor on $X$. 
We let $\mathcal O_X(D)$ denote the coherent sheaf defining $D$. 
We say that $D$ is \textit{effective} if $\mathcal O_X \subseteq \mathcal O_X(D)$. 
The set $\mathrm{WSh}(X)$ of AC divisors on $X$ 
forms naturally an additive group~\cite[Corollary~2.6]{Har94}. 
In this paper, a \textit{prime} AC divisor is an effective AC divisor 
that cannot be written as the sum of two non-zero effective AC divisors.  
A \textit{$\mathbb Q$-AC divisor} is an element of $\mathrm{WSh}(X) \otimes_{\mathbb Z} \mathbb Q$. 
Let $\Delta$ be a $\mathbb Q$-AC divisor. 
Then there are prime AC divisors $\Delta_i$ on $X$ such that 
$\Delta= \sum_{i} \delta_i \Delta_i$ for $\delta_i \in \mathbb Q$. 
We define 
$$
\lfloor \Delta \rfloor 
:= \sum_i \lfloor \delta_i \rfloor \Delta_i 
\textup{~and~}
\lceil \Delta \rceil 
:= \sum_i \lceil \delta_i \rceil \Delta_i.
$$ 
Note that $\lfloor \Delta \rfloor$ and $\lceil \Delta \rceil$ 
are not necessarily uniquely determined by $\Delta$, 
because the choice of the decomposition 
$\Delta=\sum_i \delta_i \Delta_i$ is not necessarily unique. 
We recall an example by Corti~\cite[(16.1.2)]{Kol+14}. 
Set $X:=\mathrm{Spec}\,k[x,y,z,z^{-1}]/(x^2-y^2z)$, 
$D:=(x)$ and $E:=(y)$. 
Then $D\ne E$ and $2D=2E$,
so $D= \frac{1}{2}D + \frac{1}{2}E$ as $\mathbb Q$-Cartier divisors (and so also $\mathbb Q$-AC divisors). 
\textbf{In this paper, given a $\mathbb Q$-AC divisor $\Delta$, 
we fix a decomposition $\Delta=\sum_i \delta_i \Delta_i$.
If $\Delta$ is $\mathbb Q$-Cartier, we also fix a decomposition into Cartier divisors.}
We say that $\Delta$ is \textit{effective} if $\delta_i \ge 0$ for each $i$. 
By $\Delta \ge \Delta'$ we mean $\Delta - \Delta'$ is effective. 
If $\Delta=\alpha \Delta' + \beta \Delta''$ 
for $\alpha,\beta \in \mathbb Q$ and $\mathbb Q$-AC divisors 
$\Delta'$ and $\Delta''$ whose decompositions 
$\Delta'=\sum_j \delta_j'\Delta'_j$ and 
$\Delta''=\sum_k \delta_k''\Delta''_k$ 
have already been given,
then we choose the natural decomposition 
$
\Delta=\sum_j \alpha \delta_j' \Delta_j' 
+\sum_k \beta \delta_k'' \Delta_k''.
$ 
We say that a $\mathbb Q$-AC divisor $\Delta=\sum_i \delta_i \Delta_i$ is \textit{integral} 
if $\delta_i \in \mathbb Z$ for each $i$.
 
Let $\varphi:S\to T$ be a morphism of schemes and let $T'$ be a $T$-scheme. 
We denote by $S_{T'}$ and $\varphi_{T'}:S_{T'}\to T'$ 
the fiber product $S\times_{T}T'$ and its second projection, respectively. 
For an $\mathcal O_S$-module $\mathcal G$, its pullback to $S_{T'}$ is denoted by $\mathcal G_{T'}$. 
We use the same notation for an AC or $\mathbb Q$-AC divisor if its pullback is well-defined. 
\section{Trace maps of Frobenius morphisms} \label{section:trace}
In this section, we discuss several notions defined by using the trace maps of Frobenius morphisms. 
We work over an $F$-finite field $k$ of characteristic $p>0$, 
that is, a field $k$ of characteristic $p>0$ such that the extension $k/k^p$ is finite. 
\begin{defn} \label{defn:F-sing}
Let $X$ be an equi-dimensional $k$-scheme satisfying $S_2$ and $G_1$. 
Let $\Delta$ be an effective $\mathbb Q$-AC divisor on $X$.  
We say that the pair $(X, \Delta)$ is \textit{$F$-pure} 
if for each positive integer $e$, the composite 
\begin{align*} 
\mathcal O_X \xrightarrow{{F_X^e}^\#} {F_X^e}_* \mathcal O_X 
\hookrightarrow {F_X^e}_* \mathcal O_X( \lfloor (p^e-1)\Delta \rfloor )
\end{align*}
locally splits as an $\mathcal O_X$-module homomorphism. 
%
\end{defn}
\begin{rem} \label{rem:F-pure} 
(1) When $X$ is a normal variety, Definition~\ref{defn:F-sing} is equivalent to that in \cite{HW02}. 

\noindent(2) Let $X$ be a variety satisfying $S_2$ and $G_1$ such that $iK_X$ is Cartier for $i \in \mathbb Z_{>0}$ not divisible by $p$.  Let $X^N$ denote the normalization of $X$ and let $B$ be the effective divisor corresponding to the conductor.  Miller and Schwede~\cite{MS12} have proved that if $X$ has hereditary surjective trace (see \cite[Definition~3.5]{MS12}), then the $F$-purity of $X$ is equivalent to that of the pair ($X^N$, $B$). 
\end{rem}
Let $X$ and $\Delta$ be as in Definition~\ref{defn:F-sing}. 
Suppose that $i\Delta$ is integral for some $i \in \mathbb Z_{>0}$ not divisible by $p$. 
Let $e$ be the smallest positive integer such that $(p^e-1)\Delta$ is integral. 
Let $g>0$ be an integer such that $(p^g-1)\Delta$ is integral, which is equivalent to $e|g$. 
Let $L$ be an AC divisor on $X$. 
Applying the functor $\mathcal Hom(?,\mathcal O_X(L))$
to the morphism $\mathcal O_X \to {F_X^g}_*\mathcal O_X((p^g-1)\Delta)$ 
defined by the same way as in Definition~\ref{defn:F-sing}, 
we get the morphism 
\begin{align*}
{F_X^g}_* \mathcal O_X(p^gL +(1-p^g)(K_X+\Delta))
\xrightarrow{\phi_{(X,\Delta)}^{(g)}(L)}
\mathcal O_X(L) 
\end{align*}
by Grothendieck duality. 
Throughout this paper, we denote this morphism by $\phi_{(X,\Delta)}^{(g)}(L)$. 
The pair $(X,\Delta)$ is $F$-pure if and only if 
$\phi_{(X,\Delta)}^{(g)}:=\phi_{(X,\Delta)}^{(g)}(0)$ is surjective for each $g \in \mathbb Z_{>0}$ with $e|g$. 
 
\begin{defn}[\textup{\cite[\S 3--4]{Sch14}}] \label{defn:S0}
Let $X$ be an equi-dimensional projective $k$-scheme satisfying $S_2$ and $G_1$. 
Let $\Delta$ be an effective $\mathbb Q$-AC divisor on $X$ such that $i\Delta$ is integral for some $i \in \mathbb Z_{>0}$ not divisible by $p$. 
Let $e$ be the smallest positive integer such that $(p^e-1)\Delta$ is an AC divisor. 
Let $L$ be an AC divisor on $X$. 
The $k$-vector space $S^0(X, \Delta; \mathcal O_X(L))$ 
of $H^0(X, \mathcal O_X(L))$ is defined as 
$$
\bigcap_{g>0,~e|g}
\mathrm{Image} \bigg(
H^0 \left(X, {F_X^g}_* \mathcal O_X\left( p^g L -(1-p^g)(K_X+\Delta) \right) \right)
\xrightarrow{}
H^0(X, \mathcal O_X(L) )
\bigg), 
$$
where the morphism is induced from 
$\phi^{(g)}_{(X,\Delta)}(L)$. 
When $K_X+\Delta$ is $\mathbb Q$-Cartier with index not divisible by $p$ 
and $L$ is Cartier, the above subspace is also denoted by 
$S^0(X, \sigma(X,\Delta) \otimes \mathcal O_X(L))$. 
\end{defn}
\begin{defn}[\textup{\cite[Definition~2.14]{HX15}}] \label{defn:S0f}
Let $X$ be an equi-dimensional $k$-scheme satisfying $S_2$ and $G_1$. 
Let $\Delta$ be an effective $\mathbb Q$-AC divisor on $X$ such that $i\Delta$ is integral for some $i\in \mathbb Z_{>0}$ not divisible by $p$.  
Let $e$ be the smallest positive integer such that $(p^e-1)\Delta$ is an AC divisor. 
Let $f:X\to Y$ be a proper morphism to a variety $Y$. 
Let $L$ be an AC divisor on $X$. 
The subsheaf $
S^0 f_* \left( \sigma(X,\Delta) \otimes \mathcal O_X(L) \right)
$
of $f_* \mathcal O_X(L)$ is defined as
\begin{align*}
\bigcap_{g>0,~e|g}
\mathrm{Image} \bigg(
{F_Y^g}_* f_* \mathcal O_X\left(p^g L -(1-p^g)(K_X+\Delta) \right)
\xrightarrow{}
f_* \mathcal O_X(L)
\bigg), 
\end{align*}
where the morphism is induced from 
$\phi^{(g)}_{(X,\Delta)}(L).$
\end{defn}
Considering the cohomology and base change theorem, 
one can easily check that the stalk of 
$
S^0 f_* \left( \sigma(X,\Delta) \otimes \mathcal O_X(L) \right)
$
at the generic point $\eta$ of $Y$ is isomorphic to 
$
S^0(X_{\eta},\Delta|_{X_{\eta}}; \mathcal O_{X_{\eta}}(L|_{X_{\eta}}) ). 
$
 
Next, we consider the trace maps of relative Frobenius morphisms. 
Let $X$ be an equi-dimensional $k$-scheme satisfying $S_2$ and $G_1$.
Let $\Delta$ be an effective $\mathbb Q$-AC divisor on $X$ such that 
$i\Delta$ is integral for some $i\in\mathbb Z_{>0}$ not divisible by $p$. 
Let $L$ be an AC divisor on $X$. 
Let $f:X \to Y$ be a morphism to a \textit{regular} affine variety $Y$. 
Let $U$ be the maximal open subset of $X$ such that 
$U$ is Gorenstein and $L|_U$ is Cartier. 
Replacing $X$ by $U$, we assume that $X$ is Gorenstein and $L$ is Cartier. 
Let $e$ be the smallest positive integer such that $(p^e-1)\Delta$ is an AC divisor. 
Fix $g\in \mathbb Z_{>0}$ with $e|g$. 
We now have the following commutative diagram:
\begin{align*}
\xymatrix{ X^g \ar[dr]^-{F_X^g} \ar[d]_-{F_{X/Y}^{(g)}} \ar@/_50pt/[dd]_-{f^{(g)}} &  \\
X_{Y^g} \ar[r]_-{w^{(g)}} \ar[d]_-{f_{Y^g}} & X \ar[d]^-{f} \\
Y^g \ar[r]_-{F_Y^g} & Y. 
}
\end{align*}
We consider the relative dualizing sheaf $\omega_{F_{X/Y}^{(g)}}$ of $F_{X/Y}^{(g)}$. 
Since $Y$ is regular, $F_Y$ is flat, so it is a Gorenstein morphism (\cite[V, \S 9]{Har66}). 
Then 
$
\omega_{w^{(g)}} 
\cong {f_{Y^g}}^*\omega_{F_Y^g} 
\cong {f_{Y^g}}^*\omega_Y^{1-p^g} 
$ 
by \cite[Theorem~3.6.1]{Con00}, and so 
\begin{align*}
\omega_{F_{X/Y}^{(g)}} = 
\left(F_{X/Y}^{(g)}\right)^!\mathcal O_{X_{Y^g}}
& \cong \left(F_{X/Y}^{(g)}\right)^!\left(
\omega_{w^{(g)}} \otimes \omega_{w^{(g)}}^*
\right)
\\ & \cong \left(F_{X/Y}^{(g)}\right)^!\left(
\omega_{w^{(g)}} \otimes {f_{Y^g}}^*\omega_{Y^g}^{p^g-1} 
\right)
\hspace{32pt} \textup{\footnotesize{since $\omega_{w^{(g)}}\cong {f_{Y^g}}^*\omega_Y^{1-p^g}$ }}
\\ & \cong 
\left( \left(F_{X/Y}^{(g)} \right)^! \omega_{w^{(g)}} \right)
\otimes {f^{(g)}}^*\omega_{Y^g}^{p^g-1}
\hspace{30pt} \textup{\footnotesize{by \cite[I\hspace{-1pt}I\hspace{-1pt}I, Proposition~6.9~a)]{Har66}}}
\\ & \cong 
\left( {F_X^g}^!\mathcal O_X \right)
\otimes 
{f^{(g)}}^*\omega_{Y^g}^{p^g-1}
\hspace{67pt} \textup{\footnotesize{since ${F_{X/Y}^{(g)}}^! {w^{(g)}}^! \cong {F_X^g}^!$ }}
\\ & \cong 
\omega_{X^g}^{1-p^g} \otimes {f^{(g)}}^*\omega_{Y^g}^{p^g-1} 
\cong 
\omega_{X^g/Y^g}^{1-p^g}. 
\end{align*}
Therefore, for a Cartier divisor $M$ on $X_{Y^g}$,
applying 
$
\mathcal Hom(?, \mathcal O_{X_{Y^g}}(M))
$
to the composite of 
$$
\mathcal O_{X_{Y^g}} 
\to 
{F_{X/Y}^{(g)}}_* \mathcal O_{X^g} 
\hookrightarrow
{F_{X/Y}^{(g)}}_* \mathcal O_{X^g}((p^g-1)\Delta), 
$$ 
we obtain from Grothendieck duality the following morphism: 
$$
\phi^{(g)}_{(X/Y,\Delta)}(M): 
{F_{X/Y}^{(g)}}_* \mathcal O_{X^g} \left(
{F_{X/Y}^{(g)}}^*M +(1-p^g)(K_{X/Y} +\Delta) \right) 
\to 
\mathcal O_{X_{Y^g}} (M). 
$$
Here $K_{X/Y}:=K_X -f^*K_Y$. 
Using this morphism, 
we discuss the surjectivity of $f_* \phi^{(g)}_{(X,\Delta)}(L)$.
By the above diagram, 
we obtain the commutative diagram 
$$
\xymatrix{ {F_X^g}_* \mathcal O_X((p^g-1)\Delta) & \\
{w^{(g)}}_*\mathcal O_{X_{Y^g}} \ar[u] & \mathcal O_X. \ar[ul] \ar[l]_{\alpha}
}
$$
Applying 
$
\mathcal Hom(?, \mathcal O_X(L))
$ 
to the above diagram, we get 
$$
\xymatrix{ 
{F_X^g}_* \mathcal O_X(p^gL +(1-p^g)(K_X + \Delta)) \ar[d]_{\gamma} \ar[dr]^(0.6){\phi^{(g)}_{(X,\Delta)}(L)} & \\
{w^{(g)}}_*\mathcal O_{X_{Y^g}}(L_{Y^g} +(1-p^g){f_{Y^g}}^*K_{Y} ) \ar[r]_(0.7){\beta} & \mathcal O_X(L). 
}
$$
Note that here we used the isomorphisms
$$
\mathcal Hom({w^{(g)}}_* \mathcal O_{X_{Y^g}} , \mathcal O_X )
\cong
{w^{(g)}}_* \mathcal Hom \left(\mathcal O_{X_{Y^g}}, 
{w^{(g)}}^!\mathcal O_X \right)
\cong
{w^{(g)}}_* {w^{(g)}}^!\mathcal O_X 
\cong
{w^{(g)}}_* {f_{Y^g}}^* \omega_Y^{1-p^g}, 
$$
where the first isomorphism follows from Grothendieck duality. 
Put $M := L_{Y^g} +(1-p^g){f_{Y^g}}^*K_{Y}$. 
Then $\gamma \cong {w^{(g)}}_* \phi^{(g)}_{(X/Y,\Delta)}(M)$. 
Hence, 
$$
f_* \phi_{(X,\Delta)}^{(g)}(L) 
\cong 
(f_* \beta) \circ f_* \left( {w^{(g)}}_* \phi^{(g)}_{(X/Y,\Delta)}(M) \right)
=
(f_* \beta) \circ {F_Y^g}_* \left( {f_{Y^g}}_* \phi^{(g)}_{(X/Y,\Delta)}(M) \right).
$$ 
Since $Y$ is affine, $\alpha$ splits, and hence so does $\beta$, 
which means that $f_*\beta$ is surjective. 
Thus, we see that if the morphism ${f_{Y^g}}_* \phi^{(g)}_{(X/Y,\Delta)}(M)$ is surjective (resp. the zero map), 
then so is $f_* \phi_{(X,\Delta)}^{(g)}(L)$. 

Let $l$ be an $F$-finite field that is an extension of $k$, 
let $W$ be a regular $l$-scheme, and 
let $a:W \to Y$ be a $k$-morphism. 
Set $V := X \times_Y W$ and  
let $h:V \to W$ be the second projection. 
We next consider the following commutative diagram:
$$
\xymatrix{
V^g \ar[r] \ar[d]_{F_{V/W}^{(g)}} & X^g \ar[d]^{F_{X/Y}^{(g)}} \\
V_{W^g} \ar[r]^{b} \ar[d]_{h_{W^g}} & X_{Y^g} \ar[d]^{f_{Y^g}} \\
W^g \ar[r]^{a} & Y^g. 
}
$$
Note that the squares in the diagram are cartesian. 
By \cite[Lemma~2.16]{PSZ18}, we see that 
$$ 
\phi^{(g)}_{(V/W,\Delta_V)}(M_{W^g})
\cong 
b^* \phi^{(g)}_{(X/Y,\Delta)}(M), 
$$
where $M$ is as above. 
We now assume that $a$ is flat. 
Then by the above isomorphism, we get 
$$
{h_{W^g}}_* \phi^{(g)}_{(V/W,\Delta_V)}(M_{W^g})
\cong 
a^* {f_{Y^g}}_* \phi^{(g)}_{(X/Y,\Delta)}(M). 
$$
This means that if ${f_{Y^g}}_* \phi^{(g)}_{(X/Y,\Delta)}(M)$ is surjective, 
then so is ${h_{W^g}}_* \phi^{(g)}_{(V/W,\Delta_V)}(M_{W^g})$. 
The converse holds when $a$ is surjective. 
In particular, we get the following lemma:
\begin{lem} \label{lem:S0_alg} \samepage
Let $X$ be an equi-dimensional projective $k$-scheme satisfying $S_2$ and $G_1$. 
Let $\Delta$ be an effective $\mathbb Q$-AC divisor on $X$ such that $i\Delta$ is integral for some $i\in\mathbb Z_{>0}$ not divisible by $p$. 
Let $L$ be an AC divisor on $X$. 
Let $\overline k$ be the algebraic closure of $k$. 
If 
$$
S^0\left(X_{\overline k},\Delta_{\overline k}; \mathcal O_{X_{\overline k}}(L_{\overline k}) \right)
=H^0\left(X_{\overline k}, \mathcal O_{X_{\overline k}}(L_{\overline k}) \right), 
$$
then 
$$
S^0\left(X, \Delta; \mathcal O_X(L) \right)
=H^0\left(X, \mathcal O_X(L) \right). 
$$
\end{lem}
\begin{lem} \label{lem:f-ample}
Let $X$ be an equi-dimensional projective $k$-scheme satisfying $S_2$ and $G_1$. 
Let $\Delta$ be an effective $\mathbb Q$-AC divisor on $X$ such that 
\begin{itemize}
\item $K_X +\Delta$ is $\mathbb Q$-Cartier, 
\item $i\Delta$ is integral for some $i\in\mathbb Z_{>0}$ not divisible by $p$, 
\item $(X, \Delta)$ is $F$-pure. 
\end{itemize}
Let $f:X\to Y$ be a projective morphism to a variety $Y$. 
Let $L$ be an $f$-ample Cartier divisor on $X$. 
Then there exists an integer $m_0$ such that 
$$
S^0 f_* \left( \sigma(X,\Delta) \otimes \mathcal O_X(mL+N) \right)
=f_* \mathcal O_X(mL +N) 
$$
for each $m \ge m_0$ and every $f$-nef Cartier divisor $N$ on $X$. 
\end{lem}
\begin{proof}
The proof is basically the same as that of \cite[Corollary~2.23]{Pat14}. 
Let $e$ be the smallest positive integer such that $(p^e-1)\Delta$ is integral.
We prove that there is an $m_0\in\mathbb Z_{>0}$ such that 
$$
f_*\left(\phi^{(g)}_{(X,\Delta)}(mL+N)\right):
f_*{F_X^g}_*\mathcal O_X\left(p^g(mL+N) +(1-p^g)(K_X+\Delta)\right)
\to
f_*\mathcal O_X(mL+N)
$$
is surjective for each $g\in\mathbb Z_{>0}$ with $e|g$, 
for each $m\ge m_0$ and for every $f$-nef Cartier divisor $N$ on $X$. 
By definition, $\phi^{(g)}_{(X,\Delta)}(mL+N)$ is decomposed into  
\begin{align*}
\phi^{(g)}_{(X,\Delta)}:
& {F_X^g}_*\mathcal O_X(p^g(mL+N) +(1-p^g)(K_X+\Delta))
\\ \xrightarrow{{F_X^{g-e}}_*\phi^{(e)}_{(X,\Delta)}\left(p^{g-e}(mL+N)+(1-p^{g-e})(K_X+\Delta) \right)}
& {F_X^{g-e}}_* \mathcal O_X(p^{g-e}(mL+N) +(1-p^{g-e})(K_X+\Delta))
\\ \xrightarrow{{F_X^{g-2e}}_*\phi^{(e)}_{(X,\Delta)}\left(p^{g-2e}(mL+N)+(1-p^{g-2e})(K_X+\Delta) \right)}
& {F_X^{g-2e}}_* \mathcal O_X(p^{g-2e}(mL+N) +(1-p^{g-2e})(K_X+\Delta))
\\ & \vdots
\\ \xrightarrow{{F_X^{e}}_*\phi^{(e)}_{(X,\Delta)}\left(p^{e}(mL+N)+(1-p^{e})(K_X+\Delta) \right)}
& {F_X^{e}}_* \mathcal O_X(p^{e}(mL+N) +(1-p^{e})(K_X+\Delta))
\\ \xrightarrow{\phi^{(e)}_{(X,\Delta)}\left(mL+N \right)}
& \mathcal O_X(mL+N). 
\end{align*}
Note that the above morphisms are surjective, since $(X,\Delta)$ is $F$-pure. 
Hence, it is enough to show that there is an $m_0\in \mathbb Z_{>0}$ such that 
for each $m \ge m_0$, 
for each $g\in \mathbb Z_{>0}$ with $e|g$ and 
for every $f$-nef Cartier divisor $N$ on $X$, 
\begin{align}
R^1f_*\left(\mathcal B^e (p^g(mL+N) +(1-p^g)(K_X+\Delta)) \right) =0, 
\label{cong:1} \tag{$\ast$}
\end{align}
where $\mathcal B^e(M)$ denotes the kernel of $\phi^{(e)}_{(X,\Delta)}(M)$ for every AC divisor $M$ on $X$. 
Let $d$ be the smallest positive integer such that $d(K_X+\Delta)$ is Cartier. 
Let $q_g$ and $r_g$ be respectively the quotient and remainder of the division 
of $p^g-1$ by $d$. 
Then 
\begin{align*}
& p^g(mL+N) +(1-p^g)(K_X+\Delta) 
\\ = & mL +N +(p^g-1)(mL +N - K_X-\Delta)
\\ = & mL +N +dq_g(mL +N - K_X-\Delta) +r_g(mL +N - K_X-\Delta). 
\end{align*}
Note that $r_g(mL +N - K_X-\Delta)$ is integral. 
Let $m_1$ be an integer such that $m_1L-K_X-\Delta$ is $f$-nef. 
Then $N_{m,g} := N+ dq_g(mL+N-K_X-\Delta) $ is an $f$-nef Cartier divisor for each $m\ge m_1$ and $g\in \mathbb Z_{>0}$. 
Since we have 
\begin{align*}
& \mathcal B^e (p^g(mL+N) +(1-p^g)(K_X+\Delta)) 
\\ = & \mathcal B^e \left(
mL +N +dq_g(mL +N - K_X-\Delta) +r_g(mL +N - K_X-\Delta) 
\right)
\\ = & \mathcal B^e (mL +N_{m,g}+  r_g (mL +N -K_X -\Delta) )
\\ \cong & \mathcal B^e (r_g (mL +N -K_X -\Delta))  \otimes \mathcal O_X(mL+N_{m,g}), 
\end{align*}
and $0\le r_g <d$, it follows from Keeler's relative Fujita vanishing theorem
(\cite[Theorem~1.5]{Kee03})  
that there is an $m_0 \ge m_1$ such that $(\ast)$ holds 
for each $m \ge m_0$, 
for each $g\in \mathbb Z_{>0}$ with $e|g$ and 
for every $f$-nef Cartier divisor on $X$. 
%
\end{proof}
\begin{lem} \label{lem:f-ample_rel}
Let $X$ be an equi-dimensional $k$-scheme satisfying $S_2$ and $G_1$. 
Let $\Delta$ be an effective $\mathbb Q$-AC divisor such that 
$i(K_X+\Delta)$ is Cartier for some $i\in\mathbb Z_{>0}$ not divisible by $p$. 
Let $f:X \to Y$ be a flat projective morphism to a regular variety $Y$. 
Suppose that the following conditions hold:
\begin{itemize}
\item there exists a Gorenstein open subset $U\subseteq X$ such that 
$\mathrm{codim}_{X_y}\left(X_y \setminus U_y \right) \ge 2$ for every $y\in Y$;   
\item $\mathrm{Supp}(\Delta)$ does not contain any irreducible component of any fiber of $f$; 
\item $\left(X_{\overline y}, \overline{\Delta|_{U_{\overline y}}}\right)$ is $F$-pure for every $y \in Y$, where $X_{\overline y}$ is the geometric fiber over $y$ and 
$\overline{\Delta|_{U_{{\overline y}}}}$ is the $\mathbb Q$-AC divisor 
on $X_{\overline y}$ that is the extension of $\Delta|_{U_{\overline y}}$. 
\end{itemize}
Let $A$ be an $f$-ample Cartier divisor. 
Then there exists an $m_0 \in \mathbb Z_{>0}$ such that 
$$
S^0 {f_{Y^e}}_* \left( \sigma(X_{Y^e},\Delta_{Y^e}) \otimes \mathcal O_{X_{Y^e}}(mA_{Y^e}+ N_{Y^e}) \right)
= {f_{Y^e}}_* \mathcal O_{X_{Y^e}}
(mA_{Y^e} +N_{Y^e})
$$
for each $e\ge 0$, each $m \ge m_0$ and every $f$-nef Cartier divisor $N$. 
\end{lem}
\begin{proof}
Fix $e \ge 0$. 
By the argument above, it is enough to show that the morphism
$$
{f_{Y^{d+e}}}_*\left(
\phi^{(d)}_{(X_{Y^e}/Y^e,\Delta_{Y^e})}(mA_{Y^e} +N_{Y^e})
\right)
$$
is surjective for each $d \ge 0$. 
By the argument above again, we have 
$$
{f_{Y^{d+e}}}_*\left(
\phi^{(d)}_{(X_{Y^e}/Y^e,\Delta_{Y^e})}(mA_{Y^e} +N_{Y^e})
\right)
\cong
{F_Y^e}^* {f_{Y^d}}_*\left(
\phi^{(d)}_{(X/Y,\Delta)}(mA +N)
\right)
$$
for each $d \ge 0$. The surjectivity of 
$
{f_{Y^e}}_*\left(
\phi^{(d)}_{(X/Y,\Delta)}(mA +N)
\right)
$
follows from \cite[Theorem~C]{PSZ18} or \cite[Lemma.~3.7]{Eji19p}. 
\end{proof}
\section{Positivity of coherent sheaves} \label{section:positivity}
\subsection{Base loci of coherent sheaves}
The restricted base locus and 
the augmented base locus of an $\mathbb R$-Cartier divisor 
was introduced and studied in \cite{ELMNP2}. 
These notions have been generalized to vector bundles in \cite{BKKMSU15}. 
In this subsection, we further generalize these notions to coherent sheaves
in the same way as in \cite[\S 2]{BKKMSU15} and \cite[\S 6]{FM21}.
 
Let $k$ be a field, 
let $X$ be a quasi-projective variety over $k$, 
and let $\mathrm{sp}(X)$ denote the underlying topological space of $X$. 
Let $\mathcal F$ be a coherent sheaf on $X$. 
The \textit{base locus} $\mathrm {Bs}(\mathcal F)$ of $\mathcal F$ is defined as 
the subset of $\mathrm {sp}(X)$ consisting of points at which 
$\mathcal F$ is \textit{not} generated by its global sections. 
If 
$$
\varphi:
H^0(X,\mathcal F)\otimes \mathcal O_X 
\to \mathcal F 
$$ 
denotes the evaluation map, then
$ 
\mathrm{Bs}(\mathcal F) 
=\mathrm{Supp}( \mathrm{Coker}\,\mathrm{\varphi}) .
$ 
One can check that 
\begin{itemize}
\item 
$
\mathrm{Bs}(\mathcal F \otimes \mathcal G) 
\subseteq \mathrm{Bs}(\mathcal F) 
\cup \mathrm{Bs}(\mathcal G)
$
for a coherent sheaf $\mathcal G$, 
\item 
$
\mathrm{Bs}(\mathcal G) 
\subseteq \mathrm{Bs}(\mathcal F)
\cup \mathrm{Supp}( \mathrm{Coker}\,\mathrm{\psi})
$
for a morphism 
$\psi:\mathcal F \xrightarrow{} \mathcal G$, and 
\item
$
\mathrm{Bs}(\pi^* \mathcal F) 
\subseteq 
\pi^{-1}(\mathrm{Bs}(\mathcal F))
$
for a morphism $\pi:Y \to X$. 
\end{itemize}
Take $S \subseteq \mathrm{sp}(X)$. 
We say that $\mathcal F$ is \textit{globally generated over $S$} 
(resp. \textit{generically globally generated}) 
if $\mathrm{Bs}(\mathcal F) \cap S = \emptyset$
(resp. $\mathrm{Bs}(\mathcal F) \ne X$). 

Let $D$ be a $\mathbb Q$-Cartier divisor on $X$. 
Fix a decomposition $D=\sum_{j=1}^n d_j D_j$ into Cartier divisors, 
where $d_1, \ldots, d_n \in \mathbb Q$. 
Let $i$ be the smallest positive integer such that $id_1,\ldots, id_n \in \mathbb Z$.  
Note that $iD=\sum_{j=1}^n id_j D_j$ is a Cartier divisor. 
Then we define 
$$
\mathbb{B}(\mathcal F +D) 
:= \bigcap_{l \ge 1} \left( 
\mathrm{Bs} \left( 
{S}^{il}(\mathcal F) \otimes \mathcal O_X(ilD)
\right) \right)
\subseteq 
\mathrm{sp}(X). 
$$
Set $\mathbb B(\mathcal F) := \mathbb B(\mathcal F +0)$. 
When $\mathcal F$ is locally free, 
$\mathbb B(\mathcal F)$
is called the \textit{stable base locus} of $\mathcal F$. 
We put 
$ \mathbb B(\mathcal F -D) :=\mathbb B(\mathcal F +(-D)) $ 
and
$ \mathbb B(D) :=\mathbb B(\mathcal O_X +D).  $
Take $l,m \in \mathbb Z_{>0}$. 
The natural morphisms 
\begin{align*}
{S}^l\left( 
{S}^{im}(\mathcal F) \otimes \mathcal O_Y\left(imD \right)
\right) 
& \cong
{S}^l\left( {S}^{im}(\mathcal F) \right)
\otimes \mathcal O_Y(ilmD) 
\\ & \twoheadrightarrow
{S}^{ilm}(\mathcal F) \otimes \mathcal O_Y(ilmD)
\end{align*}
imply that 
\begin{align*} 
\mathrm{Bs}\left({S}^{im}(\mathcal F) \otimes \mathcal O_Y(imD) \right)
\supseteq &
\mathrm{Bs}\left( {S}^l({S}^{im}(\mathcal F)) \otimes \mathcal O_Y(ilmD) \right)
\\ \supseteq &
\mathrm{Bs}\left( {S}^{ilm}(\mathcal F) \otimes \mathcal O_Y(ilmD) \right), 
\end{align*}
so we get an integer $n>0$ such that 
$$
\mathbb{B}(\mathcal F +D)
=\mathrm{Bs}\left( 
{S}^{in}(\mathcal F) \otimes \mathcal O_Y(inD)
\right). 
$$
One can check that
\begin{itemize}
\item 
$
\mathbb B(\mathcal F + (D + E) ) 
\subseteq 
\mathbb B(\mathcal F + D ) 
\cup
\mathbb B( E ) 
$
for a $\mathbb Q$-Cartier divisor $E$, and 
\item 
$
\mathbb B( \pi^* \mathcal F +\pi^* D)
\subseteq
\pi^{-1} \left( 
\mathbb B(\mathcal F +D)
\right)
$
for a morphism $\pi : Y\to X$ such that $\pi^* D$ can be defined. 
\end{itemize}
 
Let $A$ be a semi-ample $\mathbb Q$-Cartier divisor on $X$.
Then for every rational numbers $r_1 > r_2$, we have 
\begin{align*}
\mathbb B(\mathcal F + (D + r_1A)) 
& \subseteq 
\mathbb B(\mathcal F + (D + r_2A)) 
\cup
\mathbb B((r_1-r_2)A) 
\\ & =
\mathbb B(\mathcal F + (D + r_2A)) . 
\end{align*}
In this paper, we use the following notations:
\begin{align*}
\mathbb{B}^{A}_{-} (\mathcal F +D)
& := \bigcup_{r \in \mathbb Q_{>0}}
\mathbb{B}(\mathcal F +(D+rA)); 
\\ 
\mathbb{B}^{A}_{+} (\mathcal F +D)
& := \bigcap_{r \in \mathbb Q_{>0}}
\mathbb{B}(\mathcal F +(D-rA)).
\end{align*}
One can check that the following conditions hold:
\begin{itemize}
\item
$
\mathbb B^{A}_- (\mathcal F +D)
\subseteq 
\mathbb B (\mathcal F +D) 
\subseteq
\mathbb B^{A}_+ (\mathcal F +D)
$;
\item 
for every $r \in \mathbb Q_{>0}$ we have 
$$
\mathbb B^{rA}_-(\mathcal F +D) = \mathbb B^A_-(\mathcal F +D) 
\textup{\quad and \quad}
\mathbb B^{rA}_+(\mathcal F +D) = \mathbb B^A_+(\mathcal F +D) ;
$$
\item
for a sequence $r_1,r_2,\ldots \in \mathbb Q_{>0}$
converging to $0$, we have 
$$
\mathbb B^A_-(\mathcal F +D) 
= \bigcup_{j \ge 1} \mathbb B\left(
\mathcal F + (D +r_j A) 
\right);
$$
\item 
there is $r \in \mathbb Q_{>0}$ such that for every $r' \in (0, r] \cap \mathbb Q$, 
$$
\mathbb B_+^A(\mathcal F +D) 
=\mathbb B (\mathcal F + (D -r'A)). 
$$
\end{itemize}
\begin{lem} \label{lem:AD}
Let $X$, $\mathcal F$ and $D$ be as above. 
Let $A$ and $B$ be semi-ample $\mathbb Q$-Cartier divisors on $X$. 
Then 
\begin{align*}
\mathbb B^{A}_- (\mathcal F +D) 
& \subseteq 
\mathbb B^{B}_- (\mathcal F +D) 
\cup 
\mathbb B_+^B (A) \textup{, and} \\
\mathbb B^{B}_+ (\mathcal F +D) 
& \subseteq 
\mathbb B^{A}_+ (\mathcal F +D) 
\cup 
\mathbb B_+^B(A). 
\end{align*}
\end{lem}
\begin{proof}
Replacing $B$ by $\varepsilon B$ for some small $\varepsilon  \in \mathbb Q_{>0}$, 
we may assume that $\mathbb B^B_+(A) = \mathbb B(A -B)$. 
Then for every $r \in \mathbb Q_{>0}$, 
\begin{align*}
\mathbb B(\mathcal F +D+rA) 
\subseteq 
\mathbb B(\mathcal F +D+rB) 
\cup \mathbb B(rA-rB)
=
\mathbb B(\mathcal F +D+rB) 
\cup \mathbb B^B_+(A), 
\end{align*}
which proves the first assertion. 
By an argument similar to the above, one can prove the second assertion. 
\end{proof}

We define 
$
\mathbb B_-(\mathcal F +D)
:=\mathbb B_-^A(\mathcal F +D)
$
and 
$
\mathbb B_+(\mathcal F +D)
=\mathbb B_+^A(\mathcal F +D)
$
for an ample $\mathbb Q$-Cartier divisor $A$ on $X$. 
This definition is independent of the choice of $A$.
Indeed, for another ample $\mathbb Q$-Cartier divisors $B$ on $X$, 
we have $\mathbb B_+^A(B)=\mathbb B_+^B(A)=\emptyset$, 
so Lemma~\ref{lem:AD} shows that 
$$
\mathbb B_-^A(\mathcal F +D)
=\mathbb B_-^B(\mathcal F +D)
\textup{\quad and \quad}
\mathbb B_+^A(\mathcal F +D)
=\mathbb B_+^B(\mathcal F +D). 
$$
We set $\mathbb B_-(\mathcal F) := \mathbb B_-(\mathcal F +0)$ 
and $\mathbb B_+(\mathcal F) := \mathbb B_+(\mathcal F +0)$. 
When $\mathcal F$ is locally free, 
we call $\mathbb B_-(\mathcal F)$ 
the \textit{restricted base locus} 
(or \textit{diminished base locus}) of $\mathcal F$. 
We also call $\mathbb B_+(\mathcal F)$
the \textit{augmented base locus} of $\mathcal F$. 
We also set $\mathbb B_-(D):=\mathbb B_-(\mathcal O_X +D)$
and $\mathbb B_+(D):=\mathbb B_+(\mathcal O_X +D)$, 
which coincide respectively with the usual ones. 

As a corollary of Lemma~\ref{lem:AD}, we get the following:
\begin{cor} \label{cor:AD}
Let $X$, $\mathcal F$ and $D$ be as above. 
Let $L$ be a semi-ample $\mathbb Q$-Cartier divisor on $X$. 
Then the following hold:
\begin{itemize}
\item[\rm (1)] 
$\mathbb B_-(\mathcal F+D) \subseteq \mathbb B_-^L(\mathcal F+D) 
\subseteq \mathbb B(\mathcal F+D) \subseteq
\mathbb B_+^L(\mathcal F+D) \subseteq \mathbb B_+(\mathcal F+D)$; 
\item[\rm (2)] 
$\mathbb B_-^L(\mathcal F+D) \subseteq \mathbb B_-(\mathcal F+D) \cup \mathbb B_+(L)$;
\item[\rm (3)] 
$\mathbb B_+(\mathcal F+D) \subseteq \mathbb B_+^L(\mathcal F+D) \cup \mathbb B_+(L)$.
\end{itemize}
\end{cor}
\begin{eg} \label{eg:proper inclusions}
For each inclusion in Corollary~\ref{cor:AD}~(1), 
we give an example such that it is a proper inclusion. 
Let $X$ be a projective variety of positive dimension over a field $k$. 
Then $\mathbb B(-A) =X$ for any ample Cartier divisor $A$ on $X$, 
so we get from the definition of $\mathbb B_+^L$ that 
$$
\left\{
\begin{array}{ll}
\mathbb B(0) = \emptyset \subsetneq X = \mathbb B_+^L(0)
& \textup{if $L$ is an ample Cartier divisor on $X$, and }
\\
\mathbb B_+^L(0) = \emptyset \subsetneq X = \mathbb B_+(0)
& \textup{if $L=0. $} 
\end{array}
\right. 
$$
Next, suppose that $k$ is uncountable and algebraically closed, 
and let $E$ be an elliptic curve over $k$. 
Since $k$ is uncountable, there is a line bundle $\mathcal N$ on $E$ such that $\deg \mathcal N=0$ and $H^0(E, \mathcal N^m)=0$ for each $m>0$. 
Then $\mathcal N$ is nef and $\mathbb B(\mathcal N) = E$, so 
$$
\left\{
\begin{array}{ll}
\mathbb B_-(\mathcal N) = \emptyset \subsetneq E = \mathbb B_-^L(\mathcal N)
& \textup{if $L=0$, and}
\\
\mathbb B_-^L(\mathcal N) = \emptyset \subsetneq E = \mathbb B(\mathcal N)
& \textup{if $L$ is an ample Cartier divisor on $E$.} 
\end{array}
\right. 
$$
\end{eg}
\begin{eg} \label{eg:proper inclusions2}
We give an example satisfying 
$\mathbb B_-^L(\mathcal F)\ne\mathbb B_-^M(\mathcal F)$ 
when $L$ and $M$ are neither ample nor linearly trivial. 
Let $k$, $E$ and $\mathcal N$ be as in the latter part of Example~\ref{eg:proper inclusions}. 
Let $C$ be a smooth projective curve over $k$. 
Let $A_C$ (resp. $A_E$) be an ample divisor on $C$ (resp. $E$). 
Put $X:=C\times_k E$. 
Let $c:X\to C$ and $e:X\to E$ denote the natural projections. 
Set 
$$
L:=c^*A_C, \quad 
M:=e^*A_E \textup{\quad and \quad} 
\mathcal F := e^*\mathcal N. 
$$  
In this setting, we prove that 
$$
\mathbb B_-^L(\mathcal F)=X \textup{\quad and \quad}
\mathbb B_-^M(\mathcal F)=\emptyset. 
$$
It is enough to show that 
$\mathbb B(\mathcal F +n^{-1}L) \overset{\textup{(1)}}{=} X$ and 
$\mathbb B(\mathcal F +n^{-1}M) \overset{\textup{(2)}}{=} \emptyset$ 
for each $n>0$. Equality~(1) follows from that we have 
\begin{align*}
H^0\left(X, S^{ln}(\mathcal F) \otimes \mathcal O_X(lL) \right)
\cong H^0\left(X, \left(e^*\mathcal N^{ln} \right) 
\otimes c^*\mathcal O_C(lA_C) \right)
&\cong H^0\left(E, \mathcal N^{ln} \right) \otimes H^0(C, lA_C) 
\\ & = 0
\end{align*}
for each $l>0$ by the choice of $\mathcal N$. 
Equality~(2) is obtained from that 
$$
\left( S^{ln}\mathcal F \right) \otimes \mathcal (lM)
\cong \left( e^*\mathcal N^{ln} \right) \otimes e^*\mathcal O_E(lA_E)
\cong e^* \left(\mathcal N^{n} (A_E) \right)^l
$$
is globally generated for each $l\gg0$. 
Note that $\mathcal N^n(A_E)$ is an ample line bundle. 
 
Furthermore, putting $\mathcal G := \mathcal O_X(M)$, 
we get from an argument similar to the above that 
$$
\mathbb B_+^L(\mathcal G)=X \textup{\quad and \quad}
\mathbb B_+^M(\mathcal G)=\emptyset. 
$$
\end{eg}
In this paper, we use the following terminology, 
which is a natural extension of the positivity conditions of a $\mathbb Q$-Cartier divisor on a projective variety. 
\begin{defn} \label{defn:pe and big}
Let $X$ be a quasi-projective variety and let $D$ be a $\mathbb Q$-Cartier divisor on $X$. 
We say that $D$ is \textit{pseudo-effective} (resp. \textit{big}) 
if $\mathbb B_-(D) \ne X$ (resp. $\mathbb B_+(D) \ne X$). 
\end{defn}
In the rest of this subsection 
we prove several lemmas on base loci, 
which are used in Sections~\ref{section:invariant} and~\ref{section:main}. 
\begin{lem} \label{lem:2}
Let $X$, $D$ and $\mathcal F$ be as above. 
Let $g:X\to Z$ be a projective morphism
to a quasi-projective variety $Z$, 
let $H$ and $D$ be $\mathbb Q$-Cartier divisors on $Z$, 
and assume that $H$ is ample. 
Then 
\begin{itemize}
\item[$(1)$]
$
\mathbb B_+(g^* H)
=\left\{x \in \mathrm{sp}(X) \middle| \dim g^{-1}(g(x)) \ge 1 \right\} 
$, 
\item[$(2)$]
$
\mathrm{Bs}(\mathcal F)
\subseteq
g^{-1}(\mathrm{Bs}(g_*\mathcal F))
\cup \mathbb B_+(g^* H)
$, and
\item[$(3)$]
for a semi-ample 
$\mathbb Q$-Cartier divisor $L$ on $Z$, 
\begin{align*}
\mathbb B_-^{g^*L}(\mathcal F + g^*D)
& \subseteq
g^{-1} \left( \mathbb B_-^L( g_*\mathcal F +D) 
\cup \mathbb B(L)
\right)
\cup 
\mathbb B_+(g^*H)
\textup{\quad and} \\
\mathbb B_+^{g^*L}(\mathcal F + g^*D)
& \subseteq
g^{-1} \left( \mathbb B_+^L(g_*\mathcal F +D) 
\cup \mathbb B(L)
\right)
\cup 
\mathbb B_+(g^*H). 
\end{align*}
\end{itemize}
\end{lem}
\begin{proof}
Let $A \ge 0$ be an ample Cartier divisor on $X$. 
Fix $l \in \mathbb Z_{>0}$ such that 
$lH$ is Cartier, 
$
\mathbb B_+ (g^* H) 
=\mathbb B(g^* lH -A), 
$
and $(g_*\mathcal O_X(-A))(lH)$ is generated by its global sections. 
Take a point $x \in \mathrm{sp}(X).$
If $ g^\#_x $ is finite, then we see 
from Lemma~\ref{lem:Stein} below that 
the natural morphism 
$
g^* \left((g_*\mathcal O_X(-A))(lH) \right)
\to 
\mathcal O_X(g^*lH-A)
$
is surjective over $x$, so 
$
x \notin 
\mathrm{Bs}(g^*lH -A) \supseteq 
\mathbb B (g^*lH -A). 
$
\begin{lem} \label{lem:Stein}
Let $h:V \to W$ be a projective morphism between 
quasi-projective varieties, 
and let $\mathcal G$ be a coherent sheaf on $V$. 
Let $v$ be a point in $\mathrm{sp}(V)$ such that 
$\dim h^{-1}(h(v)) =0$. 
Then the natural morphism $\varphi: h^* h_* \mathcal G \to \mathcal G$ 
is surjective over $v$. 
\end{lem}
\begin{proof}[Proof of Lemma~\ref{lem:Stein}]
Let $V \xrightarrow{a} W' \xrightarrow{b} W$
be the Stein factorization of $h$. 
Then $b$ is an affine morphism, so the natural morphism 
$
\varphi': b^*b_* (a_* \mathcal G) \to a_* \mathcal G
$
is surjective. 
Since $\varphi$ can be decomposed as 
$
h^*h_* \mathcal G 
\xrightarrow{a^* \varphi'} 
a^* a_* \mathcal G 
\xrightarrow{\varphi''}
\mathcal G,
$
it is enough to show that $\varphi''$ is surjective, 
so we may assume that every fiber of $h$ is connected. 
Then $\{v\}=h^{-1}(h(v))$, so $h$ is finite over $h(v)$, 
and hence $\varphi$ is surjective in a neighbourhood of $v$, 
which completes the proof. 
\end{proof}
We return to the proof of Lemma~\ref{lem:2}. 
Suppose that $x \notin  \mathbb B(g^*lH -A)$. 
Then there is an effective $\mathbb Q$-Cartier divisor $E$ 
with $E \sim_{\mathbb Q} g^*lH-A$ and $x \notin \mathrm{Supp}(E)$. 
Let $F$ be an irreducible component of the fiber of $g$ over $g(x)$. 
We may assume that $F$ is not contained in $A$.
Then 
$
E|_F \sim_{\mathbb Q} (g^*lH-A)|_F \sim -A|_F, 
$
so $A|_F \sim 0$, since $F$ is projective over $g(x)$. 
This means that $\dim F =0$, so $ g^\#_x $ is finite, 
and the proof of (1) is complete.

Next, take an $l \in \mathbb Z_{>0}$ such that $lD$ is Cartier. 
We then have the morphisms 
$$
g^* \left( 
S^l(g_*\mathcal F) (lD)
\right)
\cong
S^l\left( g^*g_* \mathcal F \right) (lg^*D)
\to 
S^l(\mathcal F) (lg^*D), 
$$
and (1) tells us that the cokernel of the composite is 
supported on $\mathbb B_+(g^* H)$. 
From this, we can prove (2) and (3).
\end{proof}
\begin{lem} \label{lem:3}
Let $X$ be a quasi-projective variety, 
let $D$ be a Cartier divisor on $X$ and 
let $\mathcal F$ be a coherent sheaf on $X$. 
Then there exists an integer $n_0=n_0(\mathcal F, D)$ 
such that 
$$
\mathrm{Bs}(\mathcal F \otimes \mathcal O_X(nD)) \subseteq \mathbb B_+(D)
$$
for each $n\ge n_0$. 
\end{lem}
\begin{proof}
Let $A$ be an ample Cartier divisor on $X$. 
Fix $l,m \in \mathbb Z_{>0}$ such that
$$
\mathbb B_+(D) =\mathbb B(lD-A) =\mathrm{Bs}(m(lD-A)). 
$$
Set $\mathcal F' := \bigoplus_{0 \le r < lm}\mathcal F(rD)$. 
Take $n \in \mathbb Z_{>0}$. 
Let $q$ and $r$ denote the quotient and the remainder 
of the division of $n$ by $lm$, respectively. 
Then 
\begin{align*}
\mathrm {Bs} \left( 
\mathcal F (nD)
\right)
& =
\mathrm {Bs} \left( 
\mathcal F(rD +qmA +qm(lD-A))
\right)
\\ & \subseteq 
\mathrm {Bs} \left( 
\mathcal F (rD +qmA)
\right)
\cup
\mathrm{Bs} \left( qm(lD-A) \right)
\\ & \subseteq 
\mathrm {Bs} \left( 
\mathcal F'(qmA)
\right)
\cup
\mathbb B_+(D). 
\end{align*}
When $n\gg 0$, we have $\mathrm{Bs}(\mathcal F'(qmA))=\emptyset$,
so $\mathrm{Bs}(\mathcal F(nD)) \subseteq \mathbb B_+(D)$. 
\end{proof}
\begin{lem} \label{lem:pull_back}
Let $X$ be a quasi-projective variety, 
let $A$ and $D$ be $\mathbb Q$-Cartier divisors on $X$ and 
let $\mathcal F$ be a coherent sheaf on $X$. 
Let $\pi:Y \to X$ be a finite surjective morphism 
and let $U$ denote the maximal open subset of $X$ 
such that $\pi|_{\pi^{-1}(U)}:\pi^{-1}(U)\to U$ is flat. 
Set $C := X \setminus U$. 
Then 
\begin{align*}
\mathbb{B}(\mathcal F +(D +A))
\subseteq 
\pi \left( 
\mathbb B(\pi^* \mathcal F +\pi^*D) 
\right)
\cup \mathbb B_+(A)
\cup C. 
\end{align*}
\end{lem}
\begin{proof}
Set 
$
\mathcal G 
:= (\pi_* \mathcal O_Y)^* \otimes \pi_* \mathcal O_Y. 
$
Then there is a natural morphism $\alpha:\mathcal G \to \mathcal O_Y$ 
whose cokernel is supported on $C$. 
Take $i \in \mathbb Z_{>0}$ so that $iA$ is Cartier. 
Thanks to Lemma~\ref{lem:3}, we get $n_0 \ge 0$ such that 
$\mathrm{Bs}(\mathcal G (inA)) \subseteq \mathbb B_+(A)$
for each $n \ge n_0$. 
Fix $m \in \mathbb Z$ with $m \ge n_0$ that is divisible enough. 
Then we have a morphism 
\begin{align*}
\bigoplus^j \mathcal O_{Y} 
\xrightarrow{\beta}
\left({S}^m(\pi^*\mathcal F) \right) (m\pi^*D) 
\end{align*}
whose cokernel is supported on 
$\mathbb B (\pi^*\mathcal F +\pi^*D)$. 
Applying 
$
(\pi_*\mathcal O_{Y})^*(mA)
\otimes \pi_*(?)
$ 
to $\beta$, we get the following sequence of morphisms 
whose cokernels are supported on 
$
\pi\left(\mathbb B (\pi^*\mathcal F +\pi^*D)\right) 
\cup C:
$ 
\begin{align*}
 \bigoplus^j \mathcal G (mA)
& \xrightarrow{} 
(\pi_* \mathcal O_{Y})^*(mA)
\otimes \pi_* \big( 
\pi^* \left( {S}^m(\mathcal F)(mD) \right)
\big)
\hspace{50pt} \textup{\footnotesize{induced by $\beta$}}
\\ & \cong  
(\pi_*\mathcal O_{Y})^*(mA)
\otimes (\pi_* \mathcal O_{Y}) 
\otimes S^m(\mathcal F) (mD)
\hspace{50pt} \textup{\footnotesize{since $\pi$ is afffine}}
\\ & \cong 
\mathcal G \otimes S^m(\mathcal F) (m(D+A)) 
\\ & \to
S^m(\mathcal F) (m(D+A))
\hspace{150pt} \textup{\footnotesize{induced by $\alpha$}.}
\end{align*}
Note that since $\pi$ is affine, 
$
\mathrm{Coker}\,\pi_* \beta 
=\pi_*(\mathrm{Coker}\, \beta).
$ 
We then get that 
$$
\mathrm{Bs}\big( S^m(\mathcal F) (m(D+A)) \big)
\subseteq 
\pi\left(\mathbb B (\pi^*\mathcal F +\pi^*D)\right) 
\cup \mathbb B_+(A) \cup C. 
$$
The left-hand side is equal to $\mathbb B(\mathcal F +(D+A))$, 
since $m$ is divisible enough. 
\end{proof}
\begin{prop} \label{prop:ggg}
Let $k$ be an $F$-finite field. 
Let $W$ be a projective variety over $k$ of dimension $n$
and let $H$ be a big Cartier divisor on $W$ with $|H|$ free. 
Let $Y$ be a dense open subset of $W$
and let $\mathcal F_1,\ldots,\mathcal F_N$ and $\mathcal G$ be coherent sheaves on $Y$. 
Fix $\varepsilon, \varepsilon_1,\ldots\varepsilon_N \in \mathbb Q_{>0}$. 
Put $B:=\bigcup_{1\le i \le N}\mathbb B(\mathcal F_i -\varepsilon_i H|_Y)$. 
Then there exists a positive integer $e_0$ 
such that 
$$
\mathrm{Bs}\left(
{F_Y^e}_* \left(
S^{l_1}(\mathcal F_1) \otimes \cdots \otimes S^{l_n}(\mathcal F_n)
\otimes \mathcal G
\right)
\right)
\subseteq 
\mathbb B_+(H)
\cup B
$$
for each $e \ge e_0$ 
and each $l_1,\ldots,l_N \in \mathbb Z_{>0}$ 
with $\sum_{1 \le i \le N} \varepsilon_i l_i \ge (n+\varepsilon)p^e$. 
\end{prop}
\begin{rem}
We consider the following case: 
$Y=W$; $H$ is ample and $|H|$ is free;
$N=1$;  $\mathcal F_1=\mathcal O_W(H)$; and $\varepsilon_1 = \varepsilon =1$. 
Then, Proposition~\ref{prop:ggg} is equivalent to 
the following well-known fact:
$({F_W^e}_*\mathcal G)( (n+1)H) $ is generated by its global sections 
for each $e\gg0$.
One can use this fact to verify Fujita's freeness conjecture in positive characteristic in the case when the ample line bundle is globally generated (\cite{Har03e, Kee03}). 
This special case of Fujita’s conjecture was first proved by Smith~\cite{Smi97}.
\end{rem}
\begin{proof}
{\bf Step~1.} We first define $e_0$. 
Fix $m \in \mathbb Z_{>0}$ such that $\varepsilon_i m \in \mathbb Z$ 
and that 
$$
\mathbb B(\mathcal F_i -\varepsilon_i H|_Y) 
=\mathrm{Bs}(S^m(\mathcal F_i)(-m\varepsilon_i H|_Y))
$$
for each $i=1,\ldots,N$. 
Set 
$\mathbb S
:=\left\{(r_1,\ldots,r_N)|\textup{$0\le r_i <m$ for each $i$}\right\}
$
and 
$$
\mathcal G' 
:= \mathcal G \otimes
\bigoplus_{(r_1,\ldots,r_N)\in \mathbb S} 
\left(
\bigotimes_{1\le i \le N} S^{r_i}(\mathcal F_i) 
\right).
$$ 
Take a coherent sheaf $\mathcal G''$ on $W$ such that $\mathcal G''|_Y \cong \mathcal G'$. 
Since $|H|$ is free, there is a generically finite surjective morphism $g:W\to Z$ 
and an ample Cartier divisor $L$ on $Z$
such that $|L|$ is free and $H \sim g^*L$. 
Then, by Serre's vanishing theorem, we find $s_0 \in \mathbb Z_{>0}$ such that 
$H^j(Z, (g_*\mathcal G'')(sL))=0$
for each $1 \le j \le n$ and $s \ge s_0$. 
We define 
$
e_0:=\mathrm{min}\left\{
e\in \mathbb Z_{>0} \middle| 
\varepsilon p^e \ge s_0 + m \sum_i \varepsilon_i
\right\}.
$ 
 
\noindent {\bf Step 2.} Take $e \ge e_0$ 
and $l_1,\ldots,l_N \in \mathbb Z_{>0}$ with 
$
\sum_{i} \varepsilon_i l_i
\ge 
(n+\varepsilon)p^e. 
$
For each $i$, let $q_i$ and $r_i$ be integers 
such that $l_i=m q_i +r_i$ and $0 \le r_i < m$. 
Put $\mu:=\sum_{1\le i \le N}\varepsilon_i m q_i$ 
and $\mathcal E
:= {F_W^e}_*\left(
\mathcal G'' (\mu H)
\right).$
In this step, we prove that 
$\mathrm{Bs}(\mathcal E|_Y) \subseteq \mathbb B_+(H)$.
Since $mq_i >-m +l_i$, we have 
$$
\mu=\sum_{1\le i \le N}\varepsilon_i m q_i 
>
-\left(\sum_{1 \le i \le N} \varepsilon_im \right)
+\sum_{1 \le i \le N} \varepsilon_i l_i 
\ge
s_0 -\varepsilon p^{e_0} +(n +\varepsilon)p^e 
\ge 
s_0 +np^e. 
$$ 
so $\mu -jp^e \ge s_0$ for each $0 < j \le n$. 
Hence, we see from the projection formula that 
\begin{align*}
H^j\big(Z,\mathcal O_Z(-jL) \otimes g_* \mathcal E \big)
& \cong
H^j\big(Z,\mathcal O_Z(-jL) \otimes {F_Z^e}_*\left( (g_* \mathcal G'') (\mu L)\right) \big)
\\ & \cong 
H^j\big(Z,{F_Z^e}_*\big( (g_* \mathcal G'') ((\mu -jp^e)L)\big) \big)
\\ & \cong
H^j\big(Z,(g_* \mathcal G'') ((\mu -jp^e)L) \big)
=0, 
\end{align*}
which means that $ g_*\mathcal E $ 
is $0$-regular with respect to $\mathcal O_Z(L)$.
Thus, $g_*\mathcal E$ is generated by its global sections 
as shown in \cite[Theorem~1.8.5]{Laz04I}. 
Lemma~\ref{lem:2}~(2) then tells us that 
$
\mathrm{Bs}(\mathcal E)
\subseteq \mathbb B_+(g^*L), 
$
which proves the claim, since 
$
\mathrm{Bs}(\mathcal E|_Y) 
\subseteq \mathrm{Bs}(\mathcal E).
$
  
\noindent {\bf Step 3.}
We show the assertion. 
Put 
$$
\mathcal D 
:= \bigotimes_{1\le i \le N} 
S^{mq_i}(\mathcal F_i)(-mq_i\varepsilon_i H|_Y)
\cong 
\left(\bigotimes_{1\le i \le N} S^{mq_i}(\mathcal F_i) \right)(-\mu H|_Y). 
$$
By the definition of $\mathcal D$ and $\mathcal G'$, 
we get 
\begin{align*}
{F_Y^e}_* \left( \mathcal G'(\mu H|_Y) \otimes \mathcal D \right)
\cong &
{F_Y^e}_* \left( \mathcal G' \otimes 
\left( \bigotimes_{1\le i \le N} S^{mq_i}(\mathcal F_i) \right)
\right)
\\ \twoheadrightarrow &
{F_Y^e}_* \left( 
\mathcal G \otimes 
\left( 
\bigotimes_{1\le i \le N} S^{m q_i}(\mathcal F_i) 
\otimes S^{r_i}(\mathcal F_i) 
\right) \right)
\\ \twoheadrightarrow &
{F_Y^e}_* \left( 
\mathcal G \otimes 
\left(
\bigotimes_{1 \le i \le N} S^{l_i}(\mathcal F_i) 
\right) \right) =:\mathcal C. 
\end{align*}
Furthermore, we see from the choice of $m$ that there is a morphism 
$
\bigoplus \mathcal O_Y \to \mathcal D
$ 
whose cokernel is supported on $B$, which induces the morphism 
\begin{align*}
\bigoplus {F_Y^e}_* \left( \mathcal G'(\mu H|_Y) \right)
\cong 
{F_Y^e}_* \left( \mathcal G'(\mu H|_Y) \otimes \left( \bigoplus \mathcal O_Y \right) \right)
\to
{F_Y^e}_* \left( \mathcal G'(\mu H|_Y) \otimes \mathcal D \right)
\end{align*}
whose cokernel is supported on $B$. 
It then follows from 
$
\mathcal E|_Y 
\cong 
{F_Y^e}_* \left( \mathcal G'(\mu H|_Y) \right)
$
that 
\begin{align*}
\mathrm{Bs}(\mathcal C)
\subseteq
\mathrm{Bs}(\mathcal E|_Y) \cup B
\overset{\textup{Step 2}}{\subseteq}
\mathbb B_+(H) \cup B, 
\end{align*}
which completes the proof. 
\end{proof}
\subsection{Weak positivity}
Let $k$ be a field. A notion of weak positivity was introduced by Viehweg \cite{Vie83}. 
\begin{defn}[\textup{\cite[Variant~2.13]{Vie95}}] \label{defn:wp}
Let $Y$ be a normal quasi-projective variety 
and let $\mathcal G$ be a coherent sheaf on $Y$. 
Let $\mathcal G'$ denote the quotient of $\mathcal G$ 
by the torsion submodule 
and let $Y_1$ be the maximal open subset such that 
$\mathcal G'|_{Y_1}$ is locally free. 
Let $Y_0$ be a dense open subset of $Y_1$.  
We say that $\mathcal G$ is \textit{weakly positive over $Y_0$} if 
for every ample line bundle $\mathcal H$ on $Y$ and 
every positive integer $\alpha$, there exists a positive integer $\beta$ 
such that 
\begin{align}
S^{\alpha \beta}(\mathcal G'|_{Y_1})
\otimes 
\mathcal H|_{Y_1}^\beta
\tag{\ref{defn:wp}.1}\label{sheaf:wp}
\end{align}
is globally generated over $Y_0$. 
We simply say that $\mathcal G$ is \textit{weakly positive} if 
it is weakly positive over a dense open subset of $Y_1$. 
\end{defn}
The sheaf $\mathcal G$ is often said to be weakly positive 
if (\ref{sheaf:wp}) is generically globally generated (cf. \cite{Vie83II,Kaw85}). 
In order to distinguish this terminology from Definition~\ref{defn:wp}, 
we employ the following definition.
\begin{defn} \label{defn:psef}
Let $Y$, $\mathcal G$, $\mathcal G'$ and $Y_1$ be as in Definition~\ref{defn:wp}. 
In this paper, we say that $\mathcal G$ is \textit{pseudo-effective}
if for every ample line bundle $\mathcal H$ on $Y$ and 
every positive integer $\alpha$, there exists a positive integer $\beta$ 
such that sheaf (\ref{sheaf:wp}) is generated by its global sections at the generic point $\eta$ of $Y$. 
\end{defn}
The weak positivity and the pseudo-effectivity of coherent sheaves 
are rephrased in terms of restricted base loci. 
\begin{lem} \label{lem:relation}
Let $Y$, $\mathcal G$, $\mathcal G'$, $Y_1$ and $Y_0$ 
be as in Definition~\ref{defn:wp}. 
\\ \noindent{\rm(1)} 
The sheaf $\mathcal G$ is weakly positive if and only if $\mathbb B_-(\mathcal G'|_{Y_1})$ and $Y_0$ do not intersect. 
\\ \noindent{\rm(2)} 
If $\mathbb B_-(\mathcal G)$ and $Y_0$ do not intersect, then $\mathcal G$ is weakly positive over $Y_0$. 
\\ \noindent{\rm(3)} 
The sheaf $\mathcal G$ is pseudo-effective if and only if $\mathbb B_-(\mathcal G'|_{Y_1})$ is not equal to $Y_1$, or equivalently, $\mathbb B_-(\mathcal G'|_{Y_1})$ does not contain the generic point $\eta$ of $Y$. 
\\ \noindent{\rm(4)} 
If $\mathbb B_-(\mathcal G)$ does not contain $\eta$, then $\mathcal G$ is pseudo-effective. 
\\ \noindent{\rm(5)} 
The converse statements of {\rm(2)} and {\rm(4)} hold if $\mathcal G$ is locally free. 
\begin{align*}
\xymatrix@R=20pt@C=20pt{
\textup{$\mathcal G$ is weakly positive over $Y_0$~} 
\ar@{<=>}[r] \ar@{=>}[d]
& \textup{~$\mathbb B_-(\mathcal G'|_{Y_1})\cap Y_0 = \emptyset$~}
\ar@{<=}[r] \ar@/^15pt/[r]^-{\textup{$\mathcal G$ is locally free}} \ar@{=>}[d]
& \textup{~$\mathbb B_-(\mathcal G)\cap Y_0 = \emptyset$~} \ar@{=>}[d]
\\
\textup{ $\mathcal G$ is pseudo-effective }
\ar@{<=>}[r]
& \textup{ $\eta\notin\mathbb B_-(\mathcal G'|_{Y_1})$ }
\ar@{<=}[r] \ar@/^15pt/[r]^-{\textup{$\mathcal G$ is locally free}} 
& \textup{~$\eta\notin\mathbb B_-(\mathcal G)$~}
\\
}
\end{align*}
\end{lem}
\begin{proof}
Statements (1), (3) and (5) are obvious. 
Statements (2) and (4) follow from the inclusion $\mathbb B_-(\mathcal G|_{Y_1}) \subseteq \mathbb B_-(\mathcal G) \cap Y_1$. 
\end{proof}
The following example shows that (5) in Lemma~\ref{lem:relation} does not hold if $\mathcal G$ is not locally free. 
\begin{eg} \label{eg:blowup}
Let $Y$ be a regular projective surface, 
let $y \in \mathrm{sp}(Y)$ be a closed point, 
and let $\pi: Y' \to Y$ be the blow up of $Y$ along $y$. 
Let $\mathcal I$ be the ideal sheaf of $y$ and let $A$ be an ample Cartier divisor on $Y$. 
Take $l\gg0$ so that $\pi^*A -lE$ is not pseudo-effective, 
i.e. $\mathbb B_-(\mathcal O_{Y'}(\pi^* A -lE)) =Y'$, 
and set $\mathcal G := \mathcal I^l \otimes \mathcal O_Y(A)$. 
Then the natural map 
$
\pi^* \mathcal G \twoheadrightarrow \mathcal O_{Y'}(\pi^*A -lE)
$
shows that 
$$ 
\mathbb B_-(\mathcal G)= \mathbb B_+(\mathcal G)=Y. 
$$ 
Put $Y_1:=Y\setminus \{y\}$. 
This is the maximal open subset such that $\mathcal G|_{Y_1}$ is locally free. 
Since $\mathcal I|_{Y_1} \cong \mathcal O_{Y_1}$, we have 
$$ 
\mathbb B_-(\mathcal G|_{Y_1})= \mathbb B_+(\mathcal G|_{Y_1})=\emptyset. 
$$ 
In particular, $\mathcal G$ is weakly positive but $\mathbb B_-(\mathcal G)=Y$. 
\end{eg}
\begin{rem} \label{rem:big vector bundle}
Similarly to the argument in this subsection, 
one can discuss the bigness of coherent sheaves, 
using $\mathbb B_+$ instead of $\mathbb B_-$. 
We say that a coherent sheaf $\mathcal F$ on a projective variety $X$ is \textit{V-big} (or \textit{Viehweg-big}) if $\mathbb B_+(\mathcal F) \ne X$ (\cite[Definition~6.1]{BKKMSU15}). 
This condition is stronger than the one that the tautological bundle $\mathcal O_{\mathbb P(\mathcal F)}(1)$ is big (\cite[Examples~1.7 and~1.8.]{Jab09}). 
We do not treat the above notions further, since they are not used in this paper. 
\end{rem}
\section{An invariant of coherent sheaves} \label{section:invariant}
In this section, we introduce an invariant of coherent sheaves, 
which we use to study the positivity of coherent sheaves. 
The invariant is defined by using a morphism to a variety admitting a special endomorphism. 
Throughout this section, we work over an $F$-finite field $k$ of characteristic $p>0$. 
\begin{defn} \label{defn:epsilon}
Let $Y$ be a quasi-projective variety, 
let $H$ be a big Cartier divisor on $Y$ 
and let $S$ be a non-empty subset of $\mathrm{sp}(Y)$. 
Fix a non-negative rational number $a$. 
We say that the pair $(S,H)$ 
\textit{satisfies condition $(\ast)_{{a}}$} 
if all the following conditions hold.
\begin{itemize}
\item[I.] 
There exists a smooth projective variety $Z$ and 
a projective morphism $g :Y \to U$ 
to a dense open subset $U$ of $Z$ such that 
\begin{itemize}
\item[I-1.] 
$g^{-1}(s)$ is a finite set for every $s \in g(S)$, and 
\item[I-2.] 
$S=\bigcup_{s \in g(S)} g^{-1}(s)$.
\end{itemize}
We do not distinguish between $g: Y \to U$ and the composite 
$Y \xrightarrow{g} U \hookrightarrow Z$. 
\item[I\hspace{-1pt}I.]
There exists a big Cartier divisor $L$ on $Z$ such that 
\begin{itemize}
\item[I\hspace{-1pt}I-1.] $H \sim g^* L$ and 
\item[I\hspace{-1pt}I-2.] 
$ \mathbb B_- \left( K_Z + a L \right) \cap g(S) 
= \mathbb B_+(L) \cap g(S) 
= \emptyset$. 
\end{itemize} 
\item[I\hspace{-1pt}I\hspace{-1pt}I.]
There exists a separable finite flat endomorphism $\pi: Z \to Z$ such that 
\begin{itemize}
\item[I\hspace{-1pt}I\hspace{-1pt}I-1.]
$\pi^* L \sim qL$ for an integer $q \ge 2$, 
\item[I\hspace{-1pt}I\hspace{-1pt}I-2.]
$\pi^d$ is \'etale over a neighborhood of every point in $g(S)$ for each $d \ge 1$. 
\end{itemize}
\end{itemize} 
\end{defn}
\begin{rem} \label{rem:epsilon}
\noindent(1) We note that $q$ in I\hspace{-1pt}I\hspace{-1pt}I does not stand for a power $p^e$ of the characteristic $p$. We employ the same notation as that in some papers dealing with polarized endomorphisms (cf. \cite{NZ10}).  
\\ \noindent(2) Let $\eta$ denote the generic point of $Y$.  
When $S=\{\eta\}$, I-2 follows from I-1. 
\\ \noindent(3) When $S=\{\eta\}$ and $g$ is dominant,  
I\hspace{-1pt}I\hspace{-1pt}I-2 is always satisfied. 
\\ \noindent(4) Assume that $S=\{\eta\}$ and $g$ is dominant. 
Then I\hspace{-1pt}I-2 is equivalent to saying that 
$K_Z+{a} L$ is pseudo-effective. 
In particular, $(\ast)_{{a}}$ holds if ${a}$ is at least 
Fujita's invariant (or $a$-constant) 
$
a(Z,L)= \mathrm{inf} \{t >0 |\textup{$K_Z +tL$ is big} \}. 
$
\end{rem}
\begin{rem} \label{rem:epsilon2}
Let $\mathcal F$ be a coherent sheaf on $Y$ 
and let $D$ be a $\mathbb Q$-Cartier divisor on $Y$. 
Suppose that $(S,H)$ satisfies $(\ast)_{a}$. 
Take an ample Cartier divisor $B$ on $Z$. 
By Corollary~\ref{cor:AD}, we have 
$$
\mathbb B_-(\mathcal F +D)
\subseteq \mathbb B_-^{g^*B}(\mathcal F +D)
\subseteq \mathbb B_-(\mathcal F +D) \cup \mathbb B_+(g^*B).
$$
By conditions I-1 and I\hspace{-1pt}I-2, 
it follows from Lemma~\ref{lem:2} that 
$\mathbb B_+({g^*B}) \cap S = \emptyset$, so 
$$
\mathbb B_-^{g^*B}(\mathcal F +D) \cap S 
= \mathbb B_-(\mathcal F +D) \cap S. 
$$
Similarly, we can check that 
$
\mathbb B_+^{g^*B}(\mathcal F +D) \cap S 
= \mathbb B_+(\mathcal F +D) \cap S. 
$
\end{rem}
The next lemma follows immediately from Definition~\ref{defn:epsilon}:
\begin{lem} \label{lem:epsilon}
Let $Y$ be a quasi-projective variety, 
let $H$ be a big Cartier divisor on $Y$, 
and let $S$ be a subset of $\mathrm{sp}(Y)$ such that 
$(S,H)$ satisfies condition~$(\ast)_{{a}}$ 
for some ${a} \in \mathbb Q_{\ge 0}$. 
Let $f:Y'\to Y$ be a projective morphism 
such that 
$f|_{f^{-1}(V)} :f^{-1}(V) \to V$ is finite 
for some open subset $V \subseteq Y$ containing $S$. 
Then $(f^{-1}(S), f^* H)$ also satisfies condition~$(\ast)_{a}$. 
\end{lem}
\begin{proof}
Conditions~I-1 and I-2 hold obviously. 
Conditions~I\hspace{-1pt}I and~I\hspace{-1pt}I\hspace{-1pt}I follow from 
$
(g \circ f)(f^{-1}(S)) \subseteq g(S). 
$
\end{proof}
\begin{eg} \label{eg:very_ample}
Fix $N \in \mathbb Z_{>0}$. 
Let $Y \subseteq \mathbb P^N$ be a subvariety and let $S \subseteq \mathrm{sp}(Y)$ be a subset. 
Let $L \subset \mathbb P^N$ be a hyperplane with $Y \not\subseteq L$
and set $H:=L|_Y$. 
We show that there exist open subsets $U_1, \ldots, U_l$ of $Y$ such that 
$S \subseteq \bigcup_{i=1}^l U_i$ 
and $(S \cap U_i, H)$ satisfies $(\ast)_{N+1}$ for each $i$. 
Note that since $g:Y\to \mathbb P^N =: Z$ is injective and $K_Z +(N+1)L \sim 0$,
conditions~I and I\hspace{-1pt}I hold. 
Fix $s \in S$.
Choose a basis 
$x_0,x_1, \ldots, x_N$ of $H^0(\mathbb P^N, \mathcal O(1))$
so that 
$$
s \notin 
B:= 
\left\{ z \in \mathrm{sp}(\mathbb P^N) \middle|
\textup{$x_i$ vanishes at $z$ for some $i$}
\right\}.
$$ 
Take $q \in \mathbb Z_{\ge 2}$ with $p\nmid q$. 
Then $x_0^q, \ldots, x_N^q$ define a separable flat endomorphism 
$\pi: \mathbb P^N \to \mathbb P^N$ of degree $q^N$
such that $\pi^* L \sim qL$. 
One can check that $\pi^d$ is \'etale over $\mathbb P^N \setminus B$ for each $d\ge1$,
so $(S \setminus B, H)$ satisfies $(\ast)_{N+1}$. 
Since $Y$ is noetherian, our claim follows. 
\end{eg}
\begin{eg} \label{eg:dimY+1}
Suppose that $k$ is an $F$-finite \textit{infinite} field. 
Let $Y$ be a projective variety of dimension $n$ 
and let $H$ be a big Cartier divisor with $|H|$ free. 
Since $k$ is infinite, one can find a free linear system $\mathfrak d \subseteq |H|$ of dimension $n$. 
Let $g:Y \to \mathbb P^n =: Z$ be the morphism defined by $\mathfrak d$
and let $L$ be a hypersurface of $Z$. 
Then $g^*L \sim H$, so $g$ is finite over $\mathrm{sp}(Z)\setminus g(\mathbb B_+(H))$ by Lemma~\ref{lem:2}.
Fix $S' \subseteq \mathrm{sp}(Z) \setminus g(\mathbb B_+(H))$
and put $S:=g^{-1}(S')$. 
Combining the argument in Example~\ref{eg:very_ample} with Lemma~\ref{lem:epsilon}, 
one can prove that there exist subsets $S_1, \ldots, S_l$ of $S$ such that 
$S =\bigcup_{i=1}^l S_i$ 
and $(S_i, H)$ satisfies~$(\ast)_{n+1}$ for each $i$. 
Note that if $\mathcal L$ is ample (i.e., if $\mathbb B_+(\mathcal L) = \emptyset$), 
we can take $S=Y$. 
\end{eg}
\begin{eg} \label{eg:Abelian_variety}
Let $A$ be an abelian variety and 
let $L$ be a symmetric ample divisor on $A$ 
(i.e., an ample divisor $L$ with $(-1_A)^*L \sim L$). 
Then $(S,L)$ satisfies~$(\ast)_0$ for every subset $S \subseteq \mathrm{sp}(A)$. 
Indeed, we have $K_A +0L \sim 0$, 
and for some $m \in \mathbb Z_{\ge 2}$ with $p \nmid m$ the morphism
$
\pi:=m_A:A\ni x \mapsto m\cdot x \in A
$
is an \'etale endomorphism with the property that $\pi^*L \sim m^2L$.
\end{eg}
\begin{eg} \label{eg:mAd}
Let $Y$ be a projective variety 
and let $g : Y \to A$ be a morphism to an abelian variety $A$ 
such that $\dim Y = \dim {g(Y)}$. 
(For example, suppose that $k$ is algebraically closed and 
let $Y$ be a normal projective variety of maximal Albanese dimension.)
Let $V \subseteq A$ be the maximal open subset such that $g$ is finite over $V$. 
Take $S' \subseteq \mathrm{sp}(V)$ and put $S := g^{-1}(S')$. 
Combining Example~\ref{eg:Abelian_variety} with Lemma~\ref{lem:epsilon}, 
we see that $(S, g^*L)$ satisfies~$(\ast)_{0}$
for every symmetric ample divisor $L$ on $A$. 
\end{eg}
\begin{eg} \label{eg:toric}
Let $Y$ be a smooth projective toric variety and set $L := -K_Y$. 
Fix $S \subseteq \mathrm{sp}(T) \setminus \mathbb B_+(-K_Y)$, 
where $T$ is the dense open subset of $Y$ isomorphic to the $n$-dimensional algebraic torus $(k^{\times})^n$.
Then $(S,L)$ satisfies~$(\ast)_1$. 
Indeed, we have $K_Y +1L =0$, and for some $q \in \mathbb Z_{\ge 2}$ with $p \nmid q$ 
the $q$-th toric Frobenius morphism 
$\pi:Y \to Y$ is \'etale over $S$ and satisfies $\pi^*L \sim qL$.
\end{eg}
We need the following notion in order to define our invariant (Definition~\ref{defn:inv}). 
\begin{defn} \label{defn:pi-base_loci}
Let $Z$ be a quasi-projective variety 
and let $\pi:Z\to Z$ be a surjective endomorphism of $Z$. 
Let $\mathcal G$ be a coherent sheaf on a dense open subset $U$ of $Z$. 
Set $U_d := \left(\pi^d\right)^{-1}(U)$ and 
$\pi^d_U:= \left( \pi^d \right)|_{U_d}: U_d \to U$ 
for each $d \ge 1$. 
Let $D$ be a $\mathbb Q$-Cartier divisor on $U$ 
such that $\left( \pi^{d_0}_U \right)^*D$ is Cartier for some $d_0$. 
Set 
$$
B_d := 
{\pi_U^d} \bigg(
\mathrm{Bs}\left(
\left(\pi^d_U \right)^*\mathcal G  
\otimes_{U_d}
\mathcal O_{U_d}\left( \left(\pi_U^d \right)^*D\right)
\right)
\bigg)
\subseteq 
\mathrm{sp}(U).
$$
We then define 
\begin{align*}
\mathbb B^{\pi}(\mathcal G +D)
:= \bigcap_{d\ge d_0} B_d
\subseteq 
\mathrm{sp}(U).
\end{align*}
\end{defn}
 
Note that for each $d\ge d_0$ we have $B_d \supseteq B_{d+1}$, so 
$ \mathbb B^{\pi}(\mathcal G +D) = B_d $ for some $d \gg d_0$. 
\begin{prop} \label{prop:B-_Bpi}
Let the notation be as in Definition~\ref{defn:pi-base_loci}. 
Let $C$ denote the set of points $z \in \mathrm{sp}(U)$ such that 
$\pi^d$ is \textit{not} flat over $z$ for some $d$. 
Let $A$ be a $\mathbb Q$-Cartier divisor on $U$. 
Then 
\begin{align*}
\mathbb B(\mathcal G +(D +A) )
\subseteq 
\mathbb B^\pi(\mathcal G +D)
\cup \mathbb B_+(A)
\cup C. 
\end{align*}
\end{prop}
\begin{proof}
This follows from Lemma~\ref{lem:pull_back} immediately. 
\end{proof}
The next invariant plays an important role in Section~\ref{section:main}
\begin{defn} \label{defn:inv}
Let $Y$ be a quasi-projective variety,
let $H$ be a big Cartier divisor on $Y$, 
and let $S$ be a non-empty subset of $\mathrm{sp}(Y)$.
Fix a non-negative rational number ${a}$. 
With notation as in Definition~\ref{defn:epsilon}, 
suppose that $(S,H)$ satisfies condition~$(\ast)_{{a}}$.
Let $\mathcal F$ be a coherent sheaf on $Y$. 
We define 
\begin{align*}
T(\mathcal F)
&:=
T_S(g,\pi,\mathcal F)
:=
\left\{
r \in \mathbb Z[q^{-1}] 
\middle|
\mathbb B^{\pi}(g_*\mathcal F -rL|_U)
\cap g(S) = \emptyset
\right\}
\textup{\quad and} \\
t(\mathcal F) 
&:=
t_S(g,\pi,\mathcal F)
:=
\mathrm{sup}\,T_S(g,\pi,\mathcal F). 
\end{align*}
We note that $g:Y \to U$ is a projective morphism 
to a dense open subset $U$ of $Z$. 
\end{defn}
\begin{rem} \label{rem:inv}
(1) 
If $r \ll 0$, then 
$
\mathbb B^{\pi}(g_* \mathcal F - rL|_U)
\subseteq 
\mathrm{Bs}(g_* \mathcal F(-rL|_U)) 
\subseteq 
Z \setminus g(S)
$
by Lemma~\ref{lem:3}, 
so $T(\mathcal F) \ne \emptyset$. 
Hence we have 
$
t(\mathcal F)
\in 
\mathbb R \cup \{+\infty \}. 
$
\\ \noindent (2)
One can easily check that 
$
T(\mathcal F) \setminus \{ t(\mathcal F) \}
=\mathbb Z[q^{-1}] \cap (-\infty, t(\mathcal F)). 
$
\end{rem}
\begin{prop} \label{prop:t}
Let the notation be as in Definition~\ref{defn:inv}. 
Let $r$ be a rational number. 
\begin{itemize}
\item[$(1)$]
If $r \le t(\mathcal F)$, 
then $\mathbb B_-(\mathcal F -rH) \cap S = \emptyset.$
\item[$(2)$]
If $r < t(\mathcal F)$, 
then $\mathbb B_+(\mathcal F -rH) \cap S = \emptyset.$
\end{itemize}
\end{prop}
Note that we have 
$
\mathbb B_-^{g^*B}(? + ?') \cap S =\mathbb B_-(? + ?') \cap S
$
and 
$
\mathbb B_+^{g^*B}(? + ?') \cap S =\mathbb B_+(? + ?') \cap S
$
for some ample Cartier divisor $B$ on $Z$
as explained in Remark~\ref{rem:epsilon2}. 
\begin{proof}
We use the same notation as in Definition~\ref{defn:epsilon}. 
We first prove (1). 
Since $g|_{g^{-1}(V)}:g^{-1}(V) \to V$ is finite 
for some open subset $V \subseteq Z$ 
by 
Definition~\ref{defn:epsilon}, 
we see from Lemma~\ref{lem:2}~(1) and~(3) that 
it is enough to show that 
$
\mathbb B_-(g_* \mathcal F -rL|_U) \cap g(S) = \emptyset. 
$ 
Fix $\alpha \in \mathbb Q_{>0}$ and an ample Cartier divisor $B$ on $U$. 
We show that 
$$
\mathbb B(g_* \mathcal F +(-rL|_U +\alpha B)) \cap g(S) = \emptyset.
$$
Take $\beta \in \mathbb Q_{>0}$ so that 
$\alpha B -\beta L|_U$ is ample 
and $r -\beta \in \mathbb Z[q^{-1}].$ 
Note that $r -\beta \in T(\mathcal F)$. 
Applying Proposition~\ref{prop:B-_Bpi} with $A:=\alpha B-\beta L|_U$, 
we obtain
\begin{align*}
\mathbb B(g_*\mathcal F +(-rL|_U +\alpha B) ) 
& =
\mathbb B \big(g_*\mathcal F +( (-r +\beta)L|_U +\underbrace{\alpha B -\beta L|_U}_{\hspace{15pt}=A} ) \big) 
\\ & \subseteq 
\mathbb B^{\pi}(g_*\mathcal F +(-r +\beta)L|_U)
\subseteq 
Z \setminus g(S).
\end{align*}
Note that $C$ in Proposition~\ref{prop:B-_Bpi} is empty, 
since $\pi$ is flat by definition. 
 
Next, we show~(2). 
By an argument similar to the above, 
we only need to show that 
$
\mathbb B_+(g_* \mathcal F -rL|_U) 
\cap 
g(S) 
= \emptyset. 
$
Fix $\alpha \in \mathbb Q_{>0}$ such that 
$r +\alpha \in T(\mathcal F). $
Take $\beta \in \mathbb Q_{>0}$ so that 
$
\mathbb B_+(\alpha L|_U -\beta B)
= 
\mathbb B_+(L|_U). 
$
Using Proposition~\ref{prop:B-_Bpi} with $A=\alpha L|_U-\beta B$, we get 
\begin{align*}
\mathbb B_+(g_*\mathcal F -rL|_U) 
& \subseteq 
\mathbb B(g_*\mathcal F -(rL|_U +\beta B)) 
\\ & =
\mathbb B\big(g_*\mathcal F -((r+\alpha)L|_U \underbrace{-\alpha L|_U +\beta B}_{\hspace{20pt}=-A}) \big) 
\\ & \subseteq 
\mathbb B^{\pi}(g_*\mathcal F -(r+\alpha)L|_U)
\cup \mathbb B_+(L|_U). 
\end{align*}
Since $\mathbb B_+(L|_U)\subseteq\mathbb B_+(L)$, 
our claim follows from Definitions~\ref{defn:epsilon} and~\ref{defn:inv}. 
\end{proof}
\begin{prop} \label{prop:ACM}
Let $Y$ be a dense open subset of an $n$-dimensional projective variety $W$, 
and take $S \subseteq \mathrm{sp}(Y)$. 
Let $\mathcal E$, $\mathcal F$ and $\mathcal G$ 
be coherent sheaves on $Y$. 
Let $\Lambda$ be an infinite set of positive integers. 
Fix $\delta, M \in \mathbb R_{\ge 0}$. 
Suppose that for every $e \in \Lambda$ 
there exists a positive integer ${h_e}$ with $|\delta p^e -{h_e}| \le M$ 
and a morphism 
$$
\psi^{(e)}:
{F_Y^e}_* \left( \mathcal E \otimes S^{{h_e}}(\mathcal F) \right) 
\to 
\mathcal G 
$$
that is surjective over $S$. 
Let $H$ be a big Cartier divisor on $Y$, 
let $A'$ be a big Cartier divisor on $W$ with $|A'|$ free, 
and set $A := A'|_Y$. 
Take two rational numbers $r, r'$ so that $\delta r + r' >0$. 
Then 
\begin{align*}
\mathrm{Bs}(\mathcal G (H))
\cap S
& \subseteq
\mathbb B_- (\mathcal F -r A) 
\cup
\mathbb B_- (H -(n + r')A) 
\cup
\mathbb B_+ (A). 
\end{align*}
In particular, 
$
\mathrm{Bs}(\mathcal G (H))
\cap S
\subseteq
\mathbb B (\mathcal F -r A) 
\cup
\mathbb B (H -(n + r')A) 
\cup
\mathbb B_+ (A). 
$
\end{prop}
\begin{proof}
We first show the second assertion. 
The morphism $\psi^{(e)}$ induces 
$$
\mathcal D_e :={F_Y^e}_* \left( \mathcal E \otimes S^{{h_e}}(\mathcal F) 
\otimes \mathcal O_Y(p^e H)
\right) 
\xrightarrow{
\psi^{(e)} \otimes \mathcal O_Y(H)
}
\mathcal G (H), 
$$
which is surjective over $S$, so 
$
\mathrm{Bs}( \mathcal G (H) ) \cap S
\subseteq 
\mathrm{Bs}( \mathcal D_e ) . 
$
Take $s \in \mathbb Q$ so that 
$-r' < s < \delta r$. 
Then for each $0 \ll e \in \Lambda$ we have 
\begin{align*}
{h_e}r + (n + r') p^e 
&\ge \delta p^e r -|rM| + (n + r') p^e 
\\ &= (n +r' +s)p^e +(\delta r -s) p^e -|rM| 
\\ &\underset{e \gg 0}{>} (n +r' +s)p^e. 
\end{align*}
Applying Proposition~\ref{prop:ggg} for
$$
\big( \mathcal F_1, \varepsilon_1, l_1;~ 
\mathcal F_2, \varepsilon_2, l_2;~
\varepsilon \big)
:=\big (\mathcal F, r, h_e;~ 
\mathcal O_Y(H), n+r', p^e;~
r'+s \big), 
$$
we obtain
$$
\mathrm{Bs}( \mathcal G (H) ) \cap S
\subseteq 
\mathrm{Bs}( \mathcal D_e ) 
\overset{\textup{Prop~\ref{prop:ggg}}}{\subseteq}
\mathbb B (\mathcal F -r A) 
\cup
\mathbb B (H -(n + r')A) 
\cup
\mathbb B_+(A). 
$$

We prove the first claim. 
Since the condition $\delta r+r'>0$ is open with respect to $r$ and $r'$, 
it follows from the above argument that 
$$
\mathrm{Bs}(\mathcal G(H)) \cap S \subseteq 
\mathbb B_-^A(\mathcal F-rA)
\cup
\mathbb B_-^A(H-(n+r')A)
\cup
\mathbb B_+(A),
$$
so the first claim follows from Corollary~\ref{cor:AD}~(2). 
\end{proof}
The following two lemmas are used in the proof of Proposition~\ref{prop:main}.
\begin{lem} \label{lem:conductor}
Let $f:X\to Y$ be a separable finite flat morphism 
between smooth varieties. 
Let $V \subseteq Y$ be the largest open subset over which $f$ is \'etale. 
Let $R$ denote the ramification divisor of $f$. 
Fix $e \in \mathbb Z_{>0}$. 
Then there exists a morphism 
$$
{F_{X/Y}^{(e)}}_*\mathcal O_{X^e}((1-p^e)R) \to \mathcal O_{X_{Y^e}}
$$
of coherent sheaves on $X_{Y^e}$, 
which is an isomorphism over $f_{Y^e}^{-1}(V^e)$. 
%
\end{lem}
\begin{proof}
Set $U:=f^{-1}(V)$. We have the following commutative diagram:
\begin{align*}
\xymatrix{ 
U^e \ar@{^(->}[r]  \ar[d]_{F_{U/V}^{(e)}} & 
X^e \ar[dr]^{F_{X}^e} \ar[d]_{F_{X/Y}^{(e)}} 
  & \\ 
U_{V^e} \ar@{^(->}[r] \ar[d]_{f_{V^e}} & 
X_{Y^e} \ar[r]_{w^{(e)}} \ar[d]_{f_{Y^e}} & X \ar[d]^{f} \\
V^e \ar@{^(->}[r] & 
Y^e \ar[r]_{F_Y^e} & Y. 
}
\end{align*}
Here, each square in the diagram is cartesian. 
Since $R \sim K_{X/Y}$, applying the argument after 
Definition~\ref{defn:S0f} to the morphism $f$, 
we obtain the morphisms
$$
\phi_{(X/Y,0)}^{(e)}(0) : 
{F_{X/Y}^{(e)}}_*\mathcal O_{X^e}\left( (1-p^e)R \right)
\to
\mathcal O_{X_{Y^e}}. 
$$
We see from the choice of $V$ that $F_{U/V}^{(e)}$ is an isomorphism 
(e.g., \cite[\S 2, Proposition~2~c)~2)]{SGA5XV}), 
so it follows from the definition that $
\phi_{(X/Y,0)}^{(e)}(0)$ is an isomorphism over $U_{V^e} = {f_{V^e}}^{-1}(V^e)$. 
\end{proof}
\begin{lem} \label{lem:bc}
Let $\pi: X \to Y$ be a separable finite flat morphism 
between smooth varieties. 
Let $R$ denote the ramification divisor of $\pi$ and  
let $V \subseteq Y$ be the largest open subset over which $\pi$ is \'etale. 
Let $\mathcal G$ be a coherent sheaf on $Y$ and 
let $\mathcal M$ be a line bundle on $X$. 
Then for each $e \in \mathbb Z_{>0}$, there exists a morphism 
$$
{F_X^e}_* 
\left(
\mathcal M^{p^e}
\otimes
\mathcal O_{X^e}((1-p^e)R)
\otimes 
{\pi^{(e)}}^*\mathcal G
\right)
\to 
\mathcal M \otimes \pi^* {F_Y^e}_* \mathcal G
$$
that is an isomorphism over $\pi^{-1}(V)$. 
\end{lem}
\begin{proof}
Fix $e \in \mathbb Z_{>0}$. 
We have the following commutative diagram:
\begin{align*}
\xymatrix{ 
X^e \ar[dr]^{F_{X}^e} \ar[d]_{F_{X/Y}^{(e)}} \ar@/_50pt/[dd]_{\pi^{(e)}} & \\ 
X_{Y^e} \ar[r]_{w^{(e)}} \ar[d]_{\pi_{Y^e}} & X \ar[d]^{\pi} \\
Y^e \ar[r]_{F_Y^e} & Y. 
}
\end{align*}
Here, the square in the diagram is cartesian. 
We get 
\begin{align*}
&
{F_{X}^{e}}_* 
\left(
\mathcal M^{p^e}
\otimes 
\mathcal O_{X^e}((1-p^e)R)
\otimes 
{\pi^{(e)}}^* \mathcal G
\right)
\\ & \cong 
\mathcal M
\otimes
{w^{(e)}}_* 
{F_{X/Y}^{(e)}}_* 
\left(
\mathcal O_{X^e}((1-p^e)R)
\otimes {F_{X/Y}^{(e)}}^* {\pi_{Y^e}}^* \mathcal G
\right)
\hspace{20pt} \textup{{\footnotesize projection formula}}
\\ & \cong
\mathcal M
\otimes
{w^{(e)}}_* \left( \left(
{F_{X/Y}^{(e)}}_* \mathcal O_{X^e}((1-p^e)R)
\right)
\otimes {\pi_{Y^e}}^* \mathcal G
\right)
\hspace{40pt} \textup{{\footnotesize projection formula}}
\\ & \to
\mathcal M
\otimes
{w^{(e)}}_* {\pi_{Y^e}}^* \mathcal G
\hspace{195pt} \textup{{\footnotesize Lemma~\ref{lem:conductor}}}
\\ & \cong 
\mathcal M
\otimes
\pi^* {F_Y^e}_* \mathcal G
\hspace{210pt} \textup{{\footnotesize $\pi$ is flat.}}
\end{align*}
Note that the projection formula holds for coherent sheaves if the morphism is finite. 
Set $U:=\pi^{-1}(V)$. 
Then $U={w^{(e)}}(\pi_{V^e}^{-1}(V^e))$. 
Hence, the assertion follows from Lemma~\ref{lem:conductor}. 
\end{proof}
The next proposition plays a key role in the proofs of our main theorems. 
The situation is similar to that of Proposition~\ref{prop:ACM}. 
\begin{prop} \label{prop:main}
Let $Y$ be a dense open subset of an $n$-dimensional projective variety $W$, 
let $H$ be a big Cartier divisor on $Y$, 
and take $S \subseteq \mathrm{sp}(Y)$ so that 
$(S, H)$ satisfies condition~$(\ast)_{a}$ for some rational number $a\ge0$. 
Let $\mathcal E$, $\mathcal F$ and $\mathcal G$ 
be coherent sheaves on $Y$. 
Let $\Lambda$ be an infinite set of positive integers. 
Fix $\delta, M \in \mathbb R_{\ge 0}$. 
Suppose that for every $e \in \Lambda$ 
there exists a positive integer ${h_e}$ with $|\delta p^e -{h_e}| \le M$ 
and a morphism 
$$
\psi^{(e)}:
{F_Y^e}_* \left( \mathcal E \otimes S^{{h_e}}(\mathcal F) \right) 
\to 
\mathcal G 
$$
that is surjective over $S$. Then 
\begin{itemize}
\item[$(1)$]
$
\delta \cdot t(\mathcal F) 
\le t(\mathcal G) + {a}, 
$ 
and 
\item[$(2)$]
if $H=H'|_Y$ for some big Cartier divisor $H'$ on $W$ with $|H'|$ free, then 
$
\mathrm{Bs}(\mathcal G(lH)) 
\cap 
S = \emptyset
$
for each $l \in \mathbb Z$ with 
$
\delta \cdot t(\mathcal F) + l > n. 
$ 
\end{itemize}
\end{prop}
\begin{proof}
We first prove (2). 
Set $r' := l-n$. 
Take $r \in \mathbb Q$ so that $r < t(\mathcal F)$ and $\delta r + r' >0$. 
Then 
\begin{align*}
\mathrm{Bs}( \mathcal G (lH) ) 
\cap S
\overset{\textup{Prop~\ref{prop:ACM}}}{\subseteq}
\mathbb B_-(\mathcal F -rH)
\cup 
\underbrace{\mathbb B_-(lH -(n+r')H)}_{\hspace{40pt}=\mathbb B_-(0)=\emptyset}
\cup \mathbb B_+(H). 
\end{align*}
Since $\mathbb B_-(\mathcal F-rH) \cap S = \emptyset$ by Proposition~\ref{prop:t} and 
$\mathbb B_+(H) \cap S = \emptyset$ by Definition~\ref{defn:epsilon}, 
the assertion follows. 
Next, we prove (1). 
Let the notation be as in Definition~\ref{defn:epsilon}. 
Replacing $\mathcal E$, $\mathcal F$ and $\mathcal G$ by 
$g_* \mathcal E$, $g_*\mathcal F$ and $g_* \mathcal G$, respectively, 
we may assume that $Y=U$ and $g=\mathrm{id_U}$. 
Note that $\psi^{(e)}$ can be replaced by the composite of
$$
{F_Z^e}_* \left( 
S^{{h_e}}(g_* \mathcal F) \otimes g_* \mathcal E
\right)
\to
{F_Z^e}_* g_* \left( 
S^{{h_e}}(\mathcal F) \otimes \mathcal E 
\right)
\xrightarrow{g_* \psi^{(e)}}
g_* \mathcal G, 
$$
where each morphism is surjective over $g(S)$, 
because $g|_{g^{-1}(V)}:g^{-1}(V) \to V$ is finite 
for some open subset $V$ of $Z$ containing $g(S)$. 
By Definition~\ref{defn:epsilon}, we have 
\begin{align}
\mathbb B_- \left(K_Z +{a} L \right) \cap S = \emptyset.
\tag{\ref{prop:main}.1}\label{eq:omega}
\end{align}
Let $R_{d}$ denote the ramification divisor of $\pi^d$ for each $d>0$. 
Fix ${a}' \in \mathbb Z[q^{-1}]$ with ${a}' > {a}$. 
Take ${a}'' \in \mathbb Z[q^{-1}] \cap ({a},{a}')$
and $d_0 \in \mathbb Z_{>0}$ so that 
${a}'q^{d_0}, {a}''q^{d_0} \in \mathbb Z$. 
Set $S_d := \left( \pi^d \right)^{-1}(S)$. 
We show the following claim:
\\ \noindent {\bf Claim~1.} 
There is $d_1 \ge d_0$ such that 
for each $d \ge d_1$ we have  
$$
\mathbb B \left(-R_{d} +{a}' q^d L \right)
\cap S_d
=\emptyset. 
$$
  
For each $d \ge d_0$, we see that
\begin{align*}
-R_{d} +{a}'q^d L 
& \sim
-\left(K_Z -(\pi^d)^*K_Z\right)
+({a}'-{a}'')q^d L + {a}''(\pi^d)^* L
\\ & 
\sim
-K_Z +q^d L'
+(\pi^d)^*(K_Z +{a}'' L),  \textup{\quad where $L':=({a}'-{a}'')L$.} 
\end{align*}
%
%
%
Hence, 
\begin{align*}
& \mathbb B\left( -R_{d} +{a}'q^d L \right) 
\subseteq 
B_1 \cup B_2,
\textup{ where} \\
& B_1:= \mathbb B\left( -K_Z +q^dL' \right)
\textup{ and } 
B_2:= \mathbb B\left(
(\pi^d)^* (K_Z + {a}''  L ) 
\right). 
\end{align*}
We show that $(B_1 \cup B_2) \cap S_d = \emptyset$ for every $d\gg0$. 
By Lemma~\ref{lem:3}, there is $d_1 \ge d_0$ such that 
for each $d \ge d_1$ we have 
$$
B_1 \subseteq
\mathrm {Bs} \left( -K_Z +q^{d} L' \right)
\overset{\textup{Lem~\ref{lem:3}}}{\subseteq}
\mathbb B_+(L')
=
\mathbb B_+\left( {\pi^d}^* L \right)
\overset{\textup{Lem~\ref{lem:2}}}{\subseteq}
\left( \pi^d \right) ^{-1} 
\left( \mathbb B_+(L) \right), 
$$ 
so 
$$
(B_1 \cup B_2) \cap S_d
\subseteq 
\left( \pi^d \right)^{-1} \left( \left(
\mathbb B_+(L) \cup \mathbb B\left( K_Z +{a}'' L \right)
\right)
\cap S 
\right),
$$
and 
\begin{align*}
\mathbb B\left( 
K_Z + {a}'' L 
\right)
\overset{\textup{$a''>a$}}{\subseteq }
\mathbb B^L_- \left( 
K_Z +{a} L 
\right)
\overset{\textup{Lem~\ref{lem:AD}}}{\subseteq}
\mathbb B_-\left( K_Z + {a} L \right)
\cup
\mathbb B_+(L)
\overset{\textup{(\ref{eq:omega})}}{\subseteq} 
Z \setminus S, 
\end{align*}
which proves Claim~1. 
To state the second claim, we fix the following data:
\begin{enumerate}[(i)]
\item \label{cond:1}
an ample Cartier divisor $A$ on $Z$ with $|A|$ free; 
\item \label{cond:2}
$0 < \varepsilon_1 \in \mathbb Q$ with
$ \mathbb B(L-\varepsilon_1 A) = \mathbb B_+(L); $
\item \label{cond:3}
$0 < \nu \in \mathbb Z$ such that
$\dim Z +1 \le \nu \varepsilon_1 \in \mathbb Z$;  
\item \label{cond:4}
$ 0 \ne r \in \mathbb Z[q^{-1}]$ with $-r \in T(\mathcal F)$; 
\item \label{cond:5}
$\delta' \in \mathbb Z[q^{-1}]$ such that $\delta'r > \delta r$; 
\item \label{cond:6}
$d \gg d_1$ such that ${a}' q^d$, $\delta' q^d$, $\delta' q^d r$ and $q^d r$ are integers, and 
\begin{align*}
\mathrm{Bs} \left( (\rho_d^* \mathcal F) 
\otimes 
\mathcal O_V \left( q^d r L|_{V} \right)
\right)
\cap S_d = \emptyset, 
\end{align*}
where $V:=(\pi^d)^{-1}(U)$ and $\rho_d:=(\pi^d)|_V:V\to U$;
\item \label{cond:8}
$
\mu := ({a}' +\delta' r) q^d +\nu
\in \mathbb Z. 
$
\end{enumerate}
We prove the following claim:
\\ \noindent {\bf Claim~2.} 
$
\mathrm{Bs} \left( \left( \rho_d^* \mathcal G \right) 
\otimes \mathcal O_V\left( \mu L|_{V} \right)
\right)
\cap S_d = \emptyset.
$

If this holds, then 
$
\mathbb B^{\pi}(\mathcal G +q^{-d}\mu L|_U)
\cap S = \emptyset, 
$ 
so $-q^{-d}\mu \in T(\mathcal G), $
and hence 
$$
t(\mathcal G)
\ge -q^{-d}\mu 
= -{a}' -\delta' r -q^{-d}\nu
\xrightarrow{\textup{
$d \to \infty$, 
$\delta' \to \delta$, 
${a}' \to {a}$
}}
-{a} -\delta r, 
$$
which implies $t(\mathcal G) \ge -a +\delta t(\mathcal F)$. 
To prove Claim~2, we use Proposition~\ref{prop:ggg}. Set 
\begin{align*}
& \big(
\mathcal F_1, \varepsilon_1;~ 
\mathcal F_2, \varepsilon_2;~
\mathcal F_3, \varepsilon_3 
\big)
\\ := & 
\big(
\mathcal O_V(L|_V), \varepsilon_1;~
\mathcal O_{V}(-R +{a}'q^d L|_V) , 0;~
\left(\rho_d^*\mathcal F \right) \otimes \mathcal O_V\left(q^d r L|_V \right), 0
\big), 
\end{align*}
where $\varepsilon_1$ is that in (\ref{cond:2}) and $R:=R_d|_V$.  
Then $B$ in Proposition~\ref{prop:ggg} is equal to
\begin{align*}
\underbrace{
\mathbb B\left(
L|_V-\varepsilon_1 A|_V
\right)
}_{\textup{$= \mathbb B_+(L)$ by (\ref{cond:2})}}
\cup
\underbrace{
\mathbb B\left(
\mathcal O_V(-R +{a}' q^d L|_V )
\right)
}_{\textup{$\subseteq Z\setminus S_d$ by Claim~1}}
\cup
\underbrace{
\mathbb B\left(
(\rho_d^* \mathcal F) \otimes \mathcal O_V\left(q^dr L|_V \right)
\right)
}_{\textup{$\subseteq Z\setminus S_d$ by (\ref{cond:6})}}, 
\end{align*}
so $B \cap S_d = \mathbb B_+(L)\cap S_d =\emptyset$ by Definition~\ref{defn:epsilon}. 
Take $0 \ll e \in \Lambda$. 
%
Set 
$$
(l_1;~ l_2;~ l_3)
:=
\big( \mu p^e -{a}' (p^e-1) q^d -{h_e} q^d r;~ p^e-1 ;~ {h_e}\big). 
$$ 
Then 
\begin{align*}
l_1 \overset{\textup{(\ref{cond:8})}}{=}& (({a}' +\delta' r) q^d +\nu) p^e -{a}' (p^e-1) q^d -{h_e} q^d r
\\ =& {a}'  q^d
+(\delta' p^e -{h_e})q^d r
+\nu p^e
>  (\delta' p^e -{h_e})q^d r
+\nu p^e. 
\end{align*}
We have $(\delta'p^e -{h_e})q^d r >0$, since 
$$
\delta' p^e r
= \delta p^e r
+(\delta' r -\delta r )p^e
\ge {h_e} r -|r|M
+(\underbrace{\delta' r -\delta r}_{>0})p^e
\underset{e \gg 0}{>} {h_e} r. 
$$
Hence, we get $l_1 > \nu p^e$ and 
$$
\sum_{1 \le i \le 3} \varepsilon_i l_i 
=\varepsilon_1 l_1 
> \varepsilon_1 \nu p^e 
\ge (\dim Z +1) p^e. 
$$
Proposition~\ref{prop:ggg} shows that 
$$
\mathrm{Bs} \left(
{F_V^e}_* \left(
(\rho_d^* \mathcal E) \otimes \bigotimes_{1 \le i \le 3} S^{l_i}(\mathcal F_i)
\right)
\right)
\subseteq B 
\subseteq Z \setminus S_d. 
$$
%
The sheaf 
$
{F_V^e}_* \left(
(\rho_d^*\mathcal E) \otimes \bigotimes_{1 \le i \le 3} S^{l_i}(\mathcal F_i)
\right)
$
is isomorphic to 
\begin{align*}
{F_V^e}_* \left(
\mathcal O_V \left(\mu p^e L|_V  \right)
\otimes (\rho_d^* \mathcal E) ((1-p^e) R)
\otimes S^{{h_e}}(\rho_d^* \mathcal F)
\right)
=: \mathcal D. 
\end{align*}
Here, we used the equality 
$
l_1 + l_2 ({a}'q^d) +l_3 (q^d r) = \mu p^e,
$ 
which follows from the choice of $(l_1;~l_2;~l_3)$. 
Therefore, we obtain $\mathrm{Bs}(\mathcal D) \cap S_d = \emptyset.$
Now, we have the following sequence of morphisms:
\begin{align*}
\mathcal D \cong & 
{F_V^e}_* \left(
\mathcal O_V \left(\mu p^e L|_V +(1-p^e) R \right)
\otimes {\rho_d}^* \left( 
\mathcal E \otimes S^{{h_e}}(\mathcal F)
\right)
\right)
\hspace{30pt}\textup{\footnotesize{definition of $\mathcal D$}}
\\ \xrightarrow{u} & 
\mathcal O_{V}(\mu L|_V) 
\otimes 
{\rho}_d^*{F_V^e}_* (\mathcal E \otimes S^{{h_e}}(\mathcal F))
\hspace{120pt}\textup{\footnotesize{Lemma~\ref{lem:bc}}}
\\ \xrightarrow{v} &
\mathcal O_{V}(\mu L|_V) \otimes {\rho_d}^* \mathcal G.
\hspace{195pt}\textup{\footnotesize{induced by $\psi^{(e)}$}}
\end{align*}
Then, $u$ is surjective over $S_d$, since $\rho_d$ is \'etale over $S$. 
Also, $v$ is surjective over $S_d$ because of the assumption on $\psi^{(d)}$, 
so we get
$$
\mathrm{Bs}\left(
\mathcal O_V(\mu L|_V)
\otimes 
{\rho}_d^*\mathcal G 
\right)
\cap 
S_d = \emptyset, 
$$ 
which is our claim. 
\end{proof}
\section{Positivity of direct images} \label{section:main}
In this section, using the invariant studied in Section~\ref{section:invariant}, 
we discuss the positivity of direct images of (relative) pluricanonical bundles. 
\subsection{Direct images of pluricanonical bundles} 
\label{subsection:absolute_pluricanonical}
In this subsection, we prove the main theorems of this paper. 
We work over an $F$-finite field $k$ of characteristic $p>0$. 
To begin with, we define the following notation.
\begin{defn} \label{defn:fingen}
Let $Y$ be a variety.  
Fix $S \subseteq \mathrm{sp}(Y)$. 
Let $\mathcal R = \bigoplus_{d \ge 0} \mathcal R_d$ be a graded $\mathcal O_Y$-algebra
such that each $\mathcal R_d$ is a coherent sheaf on $Y$. 
In this paper, we say that $\mathcal R$ is 
\textit{finitely generated over $S$}
if there exists a positive integer $N$ such that 
for each $d \ge 0$ the natural morphism 
$$
\bigoplus_{\substack{i_1,\ldots,i_N \ge 0 \\ \sum_{j=1}^N i_j j =d }}
\left( \bigotimes_{j=1}^{N} S^{i_j}(\mathcal R_j)
\right)
\to 
\mathcal R_d
$$
is surjective over $S$. 
If we can take $N=1$, then we say that 
$\mathcal R$ is \textit{generated by $\mathcal R_1$ over $S$}. 
Let $\mathcal M = \bigoplus_{d \ge 0} \mathcal M_d$ be a graded $\mathcal R$-module
such that each $\mathcal M_d$ is a coherent sheaf on $Y$. 
In this paper, we say that $\mathcal M$ is 
\textit{finitely generated over $S$ as $\mathcal R$-module} 
if there exists a positive integer $N$ such that 
for each $d \ge N$ the natural morphism 
$$
\bigoplus_{i=0}^{N}
\mathcal R_{d-i} \otimes \mathcal M_i
\to
\mathcal M_d
$$
is surjective over $S$. 
\end{defn}
When $S=\{\eta\}$, where $\eta$ is the generic point of $Y$, 
the $\mathcal O_Y$-algebra $\mathcal R$ is finitely generated over $S$ if and only if 
$\mathcal R_{\eta}=\bigoplus_{d \ge 0}\mathcal (R_d)_{\eta}$
is a finitely generated $k(\eta)$-algebra. 

We first prove our main result in a general setting. 
\begin{thm} \label{thm:general}
Let $X$ be a quasi-projective equi-dimensional $k$-scheme 
satisfying $S_2$ and $G_1$, 
and let $\Delta$ be an effective $\mathbb Q$-AC divisor on $X$ such that 
$i\Delta$ is integral for some $i>0$ not divisible by $p$. 
%
Let $Y$ be a dense open subset of a projective variety $W$, 
let $H$ be a big Cartier divisor on $Y$, 
and fix subsets $S \subseteq S' \subseteq \mathrm{sp}(Y)$ 
such that $(S, H)$ satisfies condition~$(\ast)_{{a}}$ 
for some rational number $a\ge0$. 
Let $f: X \to Y$ be a surjective projective morphism. 
Let $M$ and $N$ be AC divisors on $X$ 
such that 
$
\delta N \sim_{\mathbb Q} M -(K_X +\Delta) =:M'
$
for a rational number $\delta \ge 0$. 
Let $\mathcal R(N)$ and $\mathcal R(M')$ denote the $\mathcal O_Y$-algebras
$\bigoplus_{l \ge 0} f_* \mathcal O_X(lN)$
and 
$ \bigoplus_{l \ge 0} f_* \mathcal O_X(\lfloor lM' \rfloor ) $,
respectively. 
Suppose that the following conditions hold:
\begin{itemize}
\item [(\rm i)]
$\mathcal R(N)$ is generated by $f_* \mathcal O_X(N)$
over $S'$; 
\item [(\rm i\hspace{-1pt}i)]
the $\mathcal R(M')$-module
$
\bigoplus_{l \ge 0} 
f_* \mathcal O_X \left( \lfloor lM' +K_{X}+\Delta \rfloor
\right) 
$
is finitely generated over $S'$;
\item [(\rm i\hspace{-1pt}i\hspace{-1pt}i)]
the inclusion
$
S^0f_*\left( \sigma(X,\Delta) \otimes \mathcal O_X(M) \right) 
\hookrightarrow
f_*\mathcal O_X(M) 
$ 
induces an isomorphism of stalks at every point in $S'$. 
\end{itemize}
Then 
\begin{itemize}
\item [$(1)$]
$ 
\delta \cdot t(f_* \mathcal O_X(N)) 
\le 
t(f_* \mathcal O_X(M)) + {a}, 
$ and 
\item [$(2)$]
if $H=H'|_Y$ for a big Cartier divisor $H'$ on $W$ with $|H'|$ free, then 
for each $l \in \mathbb Z$ with 
$\delta \cdot t(f_* \mathcal O_X(N)) +l > \dim Y$, the sheaf
$$
f_*\mathcal O_X(M) \otimes \mathcal O_Y(lH)
$$
is generated by its global sections at every point in $S$. 
\end{itemize}
\end{thm}
\begin{proof}
Take $e\in \mathbb Z_{>0}$ so that $(p^e-1)\Delta$ is integral.
We use the morphism defined after Remark~\ref{rem:F-pure}, that is, the morphism 
\begin{align*}
\phi^{(e)}_{(X,\Delta)}(M):
{F_X^e}_*\mathcal O_X(p^e M +(1-p^e)(K_X+\Delta)) 
\to \mathcal O_X(M). 
\end{align*}
By assumption~(i\hspace{-1pt}i\hspace{-1pt}i), the push-forward 
\begin{align}
{F_Y^e}_* f_* 
\mathcal O_X(p^e M +(1-p^e)(K_X+\Delta))
\xrightarrow{
f_* \phi^{(e)}_{(X,\Delta)}(M)
}
f_* \mathcal O_X(M) 
\tag{\ref{thm:general}.1}
\label{mor:main1}
\end{align}
is surjective over $S$ for each $e \in \mathbb Z_{>0}$. 
Take $m \in \mathbb Z_{>0}$ so that 
$m M'$ is integral,
$\delta m \in \mathbb Z$ and 
$ \delta m N \sim_{\mathbb Z} m M'$.
Put $\mathcal F := f_* \mathcal O_X(N)$. 
Then for each $l\in\mathbb Z_{>0}$, the natural morphism 
\begin{align}
S^{\delta lm}(\mathcal F)
=S^{\delta lm}(f_* \mathcal O_X(N) )
\to
f_* \mathcal O_X(\delta lm N)
\cong f_* \mathcal O_X(lm M') 
\tag{\ref{thm:general}.2}
\label{mor:main2}
\end{align}
is surjective over $S$ by assumption~(i). 
Here, we put $S^{\delta lm}(\mathcal F) := f_* \mathcal O_X$ 
when $\delta=0$. 
Replacing $m$ if necessary, 
we see from assumption~(i\hspace{-1pt}i) that there is $n_0 \in \mathbb Z_{>0}$ 
such that the natural morphism
\begin{align*}
\tag{\ref{thm:general}.3}
\label{mor:main3}
f_* \mathcal O_X(lmM')
\otimes & 
f_* \mathcal O_X \left( \lfloor nM' +(K_X + \Delta) \rfloor \right)
\\ & \to 
f_* \mathcal O_X \left( \lfloor (lm+n) M' +(K_X +\Delta) \rfloor \right)
\end{align*}
is surjective over $S$ for each $l \ge 0$ and $n \ge n_0$.
Hence, for each $l \ge 0$ and $n \ge n_0$ we get the morphism 
\begin{align*}
\tag{\ref{thm:general}.4}
\label{mor:main4}
S^{\delta lm}(\mathcal F) 
\otimes & f_* \mathcal O_X(\lfloor nM' +K_X +\Delta \rfloor)
\\ & \to 
f_* \mathcal O_X(\lfloor (lm +n) M' +K_X +\Delta\rfloor)
\end{align*}
that is surjective over $S$. 
%
Let $q_e$ and $r_e$ be integers such that 
$p^e= m q_e +r_e$ and $n_0 \le r_e <m +n_0$. 
Then 
\begin{align*}
\tag{\ref{thm:general}.5}
\label{eq:main1}
p^e M +(1-p^e)(K_X+\Delta) 
=
p^e M' +K_X +\Delta 
& =
(mq_e +r_e) M' +K_X +\Delta, 
\end{align*}
and so $r_eM' +K_X+\Delta$ is integral. 
Put 
$
\mathcal G 
:= \bigoplus_{n_0 \le r < m + n_0 } 
f_* \mathcal O_X(\lfloor r M' +K_X +\Delta \rfloor).
$
We now have the following sequence of morphisms:
\begin{align*}
{F_Y^e}_* \big(
S^{\delta m q_e}(\mathcal F)
\otimes 
\mathcal G
\big)
\overset{\textup{def of $\mathcal G$}}{\twoheadrightarrow}
& {F_Y^e}_* \big(
S^{\delta m q_e}(\mathcal F)
\otimes 
f_* \mathcal O_X\left(r_eM' +K_X+\Delta \right)
\big)
\\ \overset{\textup{(\ref{mor:main4})}}{\to} &
{F_Y^e}_* f_* \mathcal O_X((mq_e +r_e) M' +K_X+\Delta )
\\ \overset{\textup{(\ref{eq:main1})}}{=} &
{F_Y^e}_* f_* \mathcal O_X(p^e M +(1-p^e)(K_X+\Delta) )
\overset{\textup{(\ref{mor:main1})}}{\to}
f_* \mathcal O_X(M). 
\end{align*}
The composite $\psi^{(e)}$ is also surjective over $S$, 
since so is each morphism. 
Set ${h_e}:=\delta m q_e$. 
Then 
$
|{h_e} -\delta p^e| 
=\delta |m q_e -p^e|
=\delta r_e
\le \delta(m +n_1) 
$
for each $e\in\mathbb Z_{>0}$,
so we can apply Proposition~\ref{prop:main}, 
which completes the proof.
\end{proof}
The next theorem can be viewed as an analog of \cite[Theorem~1.4]{PS14}. 
\begin{thm} \label{thm:main}
Let $X$, $\Delta$, $W$, $Y$, $S$, $S'$, $H$ and $f$ be as in Theorem~\ref{thm:general}. 
Suppose that 
\begin{itemize}
\item[(\rm i)]
the $\mathcal O_Y$-algebra 
$
\bigoplus_{l \ge 0} f_* \mathcal O_X(\lfloor l(K_X +\Delta) \rfloor)
$ 
is finitely generated over $S'$, and 
\item[(\rm i\hspace{-1pt}i)]
there exists an integer $m_0\ge 0$ such that 
the inclusion 
$$ 
S^0f_* \big( \sigma(X,\Delta) \otimes \mathcal O_X(m(K_X +\Delta)) \big)
\hookrightarrow
f_*\mathcal O_X(m(K_X +\Delta))
$$ 
induces an isomorphism of stalks at every point in $S'$ 
and each $m \ge m_0$ such that $m\Delta$ is integral. 
\end{itemize}
Take $m \ge m_0$ so that $m\Delta$ is integral and set 
$
\mathcal F_m := f_* \mathcal O_X(m(K_X +\Delta)). 
$
\begin{itemize}
\item[\rm(1)] Then 
$
\mathbb B_-(\mathcal F_m +{a} m H) 
\cap S = \emptyset. 
$
In particular, if $Y$ is normal, then for each integer $l \ge {a} m$, 
the sheaf 
$
\mathcal F_m
\otimes 
\mathcal O_Y(lH)
$ 
is pseudo-effective in the sense of Definition~\ref{defn:psef}. 
\item[\rm(2)]
If $H=H'|_Y$ for a big Cartier divisor $H'$ on $W$ with $|H'|$ free, 
then for each integer $l > {a}(m-1) +\dim Y$, the sheaf 
$
\mathcal F_m
\otimes 
\mathcal O_Y(lH)
$
is generated by its global sections at every point in $S$. 
\end{itemize}
\end{thm}
\begin{rem} \label{rem:when dimY=1}
When ${a} = \dim Y +1$ (as in Example~\ref{eg:dimY+1}), 
we have 
$$
l \ge {a} m
~\Leftrightarrow~
l > {a}(m-1) +\dim Y
~\Leftrightarrow~
l \ge m (\dim Y +1). 
$$
This condition on $l$ is the same as that in \cite[Theorem~1.4]{PS14}. 
\end{rem}
\begin{proof}
Let $i>0$ be the minimum integer such that $i\Delta$ is integral.
For simplicity, put 
$
t_m 
:= 
t(\mathcal F_m )
$
for each $m \in \mathbb Z_{>0}$ with $i|m$. 
Let $\mu \ge m_0$ be an integer divisible enough.
We first show $-{a} \le \mu^{-1} t_\mu$. 
Set 
$
M := N := \mu(K_X +\Delta)
$
and $\delta:= \mu^{-1}(\mu-1)$. 
Then one can check that all the assumptions 
in Theorem~\ref{thm:general} hold, 
so the theorem shows that $\delta t_{\mu} \le t_{\mu} + {a}$, 
which means that $-{a} \le \mu^{-1}t_{\mu}$. 
Next, we take $m \ge m_0$ with $i|m$. 
Put $M:=m(K_X +\Delta)$, $N := \mu(K_X +\Delta)$
and $\delta':=\mu^{-1}(m-1)$. 
Theorem~\ref{thm:general}~(1) then says that 
$\delta' t_\mu \le t_m + {a}$.
Combining this with $-{a} \le \mu^{-1}t_{\mu}$, 
we obtain that $-{a} m \le t_m$, 
so we see from Proposition~\ref{prop:t} that
$
\mathbb B_-(\mathcal F_m +lH) \cap S = \emptyset
$
for each $l \ge {a} m$. 
Furthermore, since
$
{a} (m-1) + \dim Y \ge -\delta' t_\mu +\dim Y, 
$
the second assertion follows from Theorem~\ref{thm:general}~(2).
\end{proof}
\begin{cor} \label{cor:f-ample}
Let $X$, $\Delta$, $W$, $Y$, $S$, $H$ and $f$ be as in Theorem~\ref{thm:general}. 
Let $Y_0 \subseteq Y$ be a dense open subset containing $S$
and put $X_0:=f^{-1}(Y_0)$. 
Suppose that 
\begin{itemize}
\item $K_{X_0} +\Delta|_{X_0}$ is a $\mathbb Q$-Cartier divisor that is relatively ample over $Y_0$, and
\item $(X_0, \Delta|_{X_0})$ is $F$-pure. 
\end{itemize}
Let $i$ be the smallest positive integer such that $i\Delta$ is integral and $i(K_{X_0}+\Delta|_{X_0})$ is Cartier. 
Set $\mathcal F_m := f_* \mathcal O_X(m(K_X +\Delta))$ for each $m \ge 0$ with $i|m$. 
Then there exists an integer $m_0 \ge 0$ such that the following hold.
\begin{itemize}
\item[{\rm(1)}]
The sets 
$
\mathbb B_-(\mathcal F_m +{a}m H) 
$
and $S$ do not intersect for $m\ge m_0$ with $i|m$. 
Furthermore, if $Y$ is normal and $S\subseteq Y$ is open, 
then there exists an integer $m_1 \ge m_0$ such that 
$ \mathcal F_m \otimes \mathcal O_Y(lH) $ 
is weakly positive over $S \cap Y_1$ for $m\ge m_1$ and  $l \ge {a} m$, 
where $Y_1$ is the maximal open subset of $Y$ such that $f$ is flat over $Y_1$. 
\item[{\rm(2)}]
If $H=H'|_Y$ for a big Cartier divisor $H'$ on $W$ with $|H'|$ free, 
then for each integer $l > {a}(m-1) +\dim Y$, 
the sheaf
$
\mathcal F_m
\otimes 
\mathcal O_Y(lH)
$
is generated by its global sections at every point in $S$. 
\end{itemize}
\end{cor}
\begin{proof}
Put $S':=Y_0$. Take a sufficiently large integer $m_0$. 
We only need to check that assumptions~(i) and~(i\hspace{-1pt}i) 
in Theorem~\ref{thm:main} hold true. 
Since $K_{X_0} +\Delta|_{X_0}$ is ample over $S'$, 
we see that assumption~(i) hold,
and (i\hspace{-1pt}i) follows from Lemma~\ref{lem:f-ample}. 
Hence, we can apply Theorem~\ref{thm:main}. 
Note that Definition~\ref{defn:wp} requires the local freeness of $\mathcal F_m|_{Y_1}$, which is ensured by the choice of $Y_1$. 
\end{proof}
\subsection{Direct images of relative pluricanonical bundles} 
\label{subsection:relative_pluricanonical}
In this subsection, we deal with the positivity of the direct images of relative pluricanonical bundles. We fix an infinite $F$-finite field $k$ of characteristic $p>0$. 

The following lemma is used to prove the main theorem (Theorem~\ref{thm:relative_psef}) of this subsection. 
The author learned the proof from the referee. 
\begin{lem} \label{lem:base_change}
Let $X$ be an equi-dimensional quasi-projective $k$-scheme satisfying $S_2$ and $G_1$, 
let $Y$ be a regular quasi-projective variety, 
and let $f:X\to Y$ be a surjective projective morphism. 
Let $e$ be a positive integer. 
Then $X_{Y^e}$ is also an equi-dimensional quasi-projective $k$-scheme satisfying $S_2$ and $G_1$.
\end{lem}
\begin{proof}
Since $F_Y^e$ is a universal homeomorphism, the projection $w^{(e)}:X_{Y^e}\to X$
is homeomorphic, so $X_{Y^e}$ is equi-dimensional. 
Also, since $Y$ is regular, $F_Y^e$ is a flat morphism with Gorenstein fibers, 
and hence so is the base change $w^{(e)}$ (\cite[Corollary~2']{WITO69}), 
which means that $X_{Y^e}$ satisfies $S_2$ and $G_1$ (\cite[Proposition~1~(i\hspace{-1pt}i)]{RF72}). 
\end{proof}
\begin{thm} \label{thm:relative_psef}
Let $X$ be an equi-dimensional quasi-projective $k$-scheme satisfying $S_2$ and $G_1$, 
let $Y$ be a dense normal open subset of a projective variety $W$ of dimension $n$, 
and let $\eta$ $($resp. $\overline \eta)$ be the generic $($resp. geometric generic$)$ point of $Y$. 
Let $f:X \to Y$ be a surjective projective morphism 
and let $\Delta$ be an effective $\mathbb Q$-AC divisor on $X$ 
such that $i(K_X+\Delta)$ is Cartier for some $i >0$ not divisible by $p$. 
Suppose that the following conditions hold:
\begin{itemize}
\item[{\rm (i)}] 
$
\bigoplus_{m \ge 0} 
H^0\left(X_{\overline \eta}, 
\mathcal O_{X_{\overline\eta}}
\left( \lfloor m \left(K_{X_{\overline\eta}} + \Delta|_{X_{\overline\eta}} \right) \rfloor 
\right) \right)
$
is finitely generated $k(\overline\eta)$-algebra; 
\item[{\rm (i\hspace{-1pt}i)}] 
there exists an integer $m_0 \ge 0$ such that 
$$
S^0 \left(X_{\overline \eta}, \Delta|_{X_{\overline \eta}}; 
\mathcal O_{X_{\overline \eta}} (m(K_{X_{\overline \eta}} + \Delta|_{X_{\overline \eta}})) \right)
=H^0 \left(X_{\overline \eta}, \mathcal O_{X_{\overline \eta}}(m(K_{X_{\overline \eta}} + \Delta|_{X_{\overline\eta}})) \right)
$$
for each $m \ge m_0$ such that $m\Delta$ is integral.  
\end{itemize}
Define $\mathcal G_m := f_* \mathcal O_X\left( m(K_{X}+\Delta) \right) \otimes \omega_Y^{[-m]}$ for an integer $m \ge m_0$ such that $m\Delta$ is integral. 
\begin{itemize}
\item[{\rm(1)}] 
Then $\mathcal G_m$ is pseudo-effective in the sense of Definition~\ref{defn:psef}. 
\item[{\rm(2)}]
If $Y$ is regular, then $\mathbb B_-(\mathcal G_m) \ne Y$. 
\item[{\rm(3)}]
Let $A$ be a big Cartier divisor on $W$ with $|A|$ free, 
and let $H$ be a Cartier divisor on $Y$ such that $H -n A|_Y$ is big. 
If $Y$ is regular, then the sheaf 
$$
\mathcal G_m \otimes \mathcal O_Y(K_Y +H)
$$
is generated by its global sections at $\eta$. 
\end{itemize}
\end{thm}
\begin{rem} \label{rem:weak positivity}
Statement~(1) in the theorem is a positive characteristic analog of the weak positivity theorem due to Viehweg~\cite[Theorem~I\hspace{-1pt}I\hspace{-1pt}I]{Vie83}. 
The first weak positivity theorem in positive characteristic is due to Patakfalvi~\cite[Theorem~1.1]{Pat14}. 
In \cite{Eji17}, based on Patakfalvi's techniques, the author proved a higher-dimensional version of \cite[Theorem~1.1]{Pat14}. 
\end{rem}
\begin{proof}[Proof of Theorem~\ref{thm:relative_psef}]
We first consider (1). 
Let $\mathcal G'_m$ be the torsion-free part of $\mathcal G_m$
and let $Y_1\subseteq Y$ be the maximal open subset on which $\mathcal G_m'$ is locally free. 
Let $Y_{\mathrm{reg}}$ be the regular locus of $Y$. 
Then clearly 
$
\mathbb B_-(\mathcal G_m'|_{Y_1})\cap Y_{\mathrm{reg}} 
=\mathbb B_-(\mathcal G_m'|_{Y_1\cap Y_{\mathrm{reg}}}), 
$ 
so (1) follows from (2). 
To prove (2), we use Theorem~\ref{thm:main}. 
Let $H$ be a very ample Cartier divisor on $Y$ 
and put $S:=S':=\{\eta\}$. 
As shown in Example~\ref{eg:very_ample}, 
the pair $(S, H)$ satisfies condition~$(\ast)_{a}$ 
for some $a \in \mathbb Z_{>0}$. 
Take $e \in \mathbb Z_{>0}$.  
By Lemma~\ref{lem:S0_alg}, one can check that 
$f_{Y^e}:X_{Y^e}\to Y^e$ and $\Delta_{Y^e}$ 
satisfy conditions~(i) and~(i\hspace{-1pt}i) in Theorem~\ref{thm:main}. 
Note that $X_{Y^e}$ is an equi-dimensional quasi-projective $k$-scheme satisfying $S_2$ and $G_1$ by Lemma~\ref{lem:base_change}. 
Set $L:= (am -a +\dim Y +1)H +mK_Y$. 
Since $F_Y^e$ is flat, we have 
\begin{align*}
\left( {F_Y^e}^* \mathcal G_m \right)
\otimes 
\mathcal O_{Y^e}(L) 
= & 
\left({F_Y^e}^* f_*\mathcal O_X(m(K_{X/Y}+\Delta)) \right)
\otimes
\mathcal O_{Y^e}(L)
\\ \cong & 
{f_{Y^e}}_* \mathcal O_{X_{Y^e}}\left( m(K_{X_{Y^e}/Y^e} +\Delta_{Y^e}) \right)
\otimes 
\mathcal O_{Y^e}(L) 
\\ \cong &
{f_{Y^e}}_* \mathcal O_{X_{Y^e}}\left( m(K_{X_{Y^e}} +\Delta_{Y^e}) \right) 
\otimes 
\mathcal O_{Y^e}((am -a +\dim Y +1)H), 
\end{align*}
so Theorem~\ref{thm:main}~(2) tells us that 
$
\eta\notin\mathrm{Bs}\left(
	\left( {F_Y^e}^*\mathcal G_m \right) \otimes \mathcal O_{Y^e}(L) 
\right). 
$
This means that 
$$
\eta\notin\mathbb B^{F_Y} ( \mathcal G_m  + p^{-e}L). 
$$
Let $M$ be an ample divisor on $Y$ such that $L+M$ is ample. 
Thanks to Proposition~\ref{prop:B-_Bpi}, we get
$
\eta\notin\mathbb B(\mathcal G_m  + p^{-e}(L + M)), 
$
so 
$$
\mathbb B_- (\mathcal G_m ) 
=\bigcup_{e > 0} \mathbb B(\mathcal G_m  +p^{-e}(L+M)) \not\ni \eta.
$$
We show (3). 
Let $\mu \ge m_0$ be an integer divisible enough. 
By an argument similar to that in the proof of Theorem~\ref{thm:general}, 
we have $n_0 \in \mathbb Z_{>0}$ such that for each $e \in \mathbb Z_{>0}$, 
there is the morphism 
\begin{align*}
\tag{\ref{thm:relative_psef}.1} \label{mor:a}
{F_Y^e}_* \bigg(
S^{q_e} \big( f_* \mathcal O_X(\mu(K_X +\Delta)) \big)
\otimes 
f_* \mathcal O_X(r_e(K_X +\Delta))
\bigg)
\to
f_* \mathcal O_X(m(K_X +\Delta))
\end{align*}
that is surjective over $S$, 
where $q_e$ and $r_e$ are integers such that 
$(m-1)p^e +1 = \mu q_e + r_e$ and $n_0 \le r_e < n_0 + \mu$. 
(If $m=1$, then we put $n_0:=1$, $q_e=0$ and $S^{0}(?):=f_*\mathcal O_X$.)
Note that $r_e(K_X+\Delta)$ is integral by the definition. 
For each $l\in\mathbb Z_{>0}$, we denote by $\mathcal G_{l}$ 
the sheaf $f_*\mathcal O_X(\lfloor l(K_{X/Y}+\Delta)\rfloor)$.
Taking the tensor product of (\ref{mor:a}) 
and $\omega_Y^{1-m}$, we obtain
\begin{align}
\tag{\ref{thm:relative_psef}.2} \label{mor:b}
{F_Y^e}_* \big(
S^{q_e} (\mathcal G_\mu) \otimes \mathcal G_{r_e} \otimes \omega_Y
\big)
\to 
\mathcal G_m \otimes \omega_Y
\end{align}
by the projection formula. 
Putting $\mathcal E := \bigoplus_{n_0 \le r < n_0 +\mu} \mathcal G_{r}(K_Y)$, 
we get the morphism 
$$
{F_Y^e}_* \left( S^{q_e}(\mathcal G_\mu) \otimes \mathcal E \right)
\to \mathcal G_m \otimes \omega_Y
$$
which is surjective over $S$. 
By the assumption, we find $r'\in\mathbb Q_{>0}$ such that $H-(n+r')A|_Y$ is big. 
We use Proposition~\ref{prop:ACM} with the following data:
$$
\left( \mathcal E,~ \mathcal F,~ \mathcal G;~ h_e,~ \delta;~ r,~r' \right)
:=\left( \mathcal E,~ \mathcal G_\mu,~ \mathcal G_m \otimes \omega_Y;~ q_e,~ \frac{m-1}{\mu};~ 0,~r' \right)
$$
Then $|\delta p^e -{h_e}|=\mu^{-1}|(m-1)p^e -\mu q_e| \le \mu^{-1}(n_0 +\mu)$, 
so we get
\begin{align*}
\mathrm{Bs}(\mathcal G_m(K_Y +H)) \cap S
\subseteq
\mathbb B_-(\mathcal G_\mu)
\cup 
\mathbb B( \underbrace{H -(n +r')A|_Y}_{\textup{big}}) 
\cup 
\mathbb B_+(\underbrace{A|_Y}_{\textup{big}})
\end{align*}
by Proposition~\ref{prop:ACM}. 
Hence, (2) implies that $\mathrm{Bs}(\mathcal G_m(K_Y+H)) \cap S =\emptyset$. 
\end{proof}
Next, we prove the weak positivity of the direct images of relative pluricanonical bundles, 
in the case where the geometric generic fiber is $F$-pure and has ample dualizing sheaf. 
\begin{thm} \label{thm:relative_f-ample}
Let $X$, $\Delta$, $W$, $Y$ and $f$ be as in Theorem~\ref{thm:relative_psef}. 
%
Let $U\subseteq X$ be the largest Gorenstein open subset. 
Let $Y_0\subseteq Y$ be the subset consisting of points $y\in Y$ with the following properties: 
\begin{itemize}
\item $y$ is a regular point;
\item $f$ is flat at every point in $f^{-1}(y)$; 
\item $X_y$ satisfies $S_2$ and $G_1$;
\item $\mathrm{Supp}(\Delta)$ does not contain any irreducible component of $f^{-1}(y)$;
\item $\left(X_{\overline y}, \overline{\Delta|_{U_{\overline y}}}\right)$ is $F$-pure, where $X_{\overline y}$ is the geometric fiber of $f$ over $y$ and $\overline{\Delta|_{U_{\overline y}}}$ is the $\mathbb Q$-AC divisor on $X_{\overline y}$ that is the extension of $\Delta|_{U_{\overline y}}$ to $X_{\overline y}$; 
\item $K_X+\Delta$ is $\mathbb Q$-Cartier in a neighborhood of every point in $f^{-1}(y)$;
\item $K_{X_{y}} +\overline{\Delta|_{U_{\overline y}}}$ is ample. 
\end{itemize}
Define 
$
\mathcal G_m:=f_*\mathcal O_X(m(K_{X}+\Delta)) \otimes \omega_Y^{[-m]}
$ 
for each positive integer $m \ge m_0$ with $i|m$. 
Then there exists a positive integer $m_0$ such that the following conditions hold.
\begin{itemize}
\item[\rm(0)] 
The set $Y_0$ is an open subset of $Y$.
\item[\rm(1)] 
The sheaf $ \mathcal G_m $ is weakly positive over $Y_0$ 
for each $m\ge m_0$ with $i|m$.
\item[\rm(2)] 
If $Y$ is regular, then $\mathbb B_-(\mathcal G_m) \cap Y_0 =\emptyset$ 
for each $m\ge m_0$ with $i|m$. 
\item[\rm(3)] 
Let $A$ be an ample Cartier divisor on $W$ with $|A|$ free and 
let $H$ be a Cartier divisor on $Y$ such that $H-nA|_Y$ is ample. 
If $Y$ is regular, then the sheaf
$$
\mathcal G_m \otimes \mathcal O_Y(K_Y +H)
$$
is globally generated over $Y_0$ 
for each $m \ge m_0$ with $i|m$. 
\end{itemize}
\end{thm}
\begin{proof}
For (0), one can check that each condition on $y$ is open on $Y$. 
Note that the openness of the $F$-purity of fibers follows from 
\cite[Theorem~3.28]{PSZ18}. 
We see that (1) follows from (2), applying the same argument as that of the proof of Theorem~\ref{thm:relative_psef}~(1). 
We prove (2) and (3).  By Lemma~\ref{lem:f-ample_rel}, 
there is $m_0 \in \mathbb Z_{>0}$ such that the natural inclusion 
\begin{align*}
S^0{f_{Y^e}}_* \left( \sigma(X_{Y^e},\Delta_{Y^e}) \otimes \mathcal O_{X_{Y^e}}(m(K_X+\Delta)_{Y^e}) \right)
\hookrightarrow 
{f_{Y^e}}_* \mathcal O_{X_{Y^e}}
(m(K_X+\Delta)_{Y^e})
\end{align*}
is an isomorphism over $Y_0$ 
for each $e \in \mathbb Z_{>0}$ and $m \ge m_0$ with $i|m$. 
Hence, replacing the generic point $\eta$ with a point in $Y_0$, 
we can apply the same argument as that in the proof of Theorem~\ref{thm:relative_psef}. 
\end{proof}
\subsection{Conclusions}
In this subsection, for the reader's convenience, 
we summarize the conclusions in Subsections~\ref{subsection:absolute_pluricanonical} and~\ref{subsection:relative_pluricanonical}
in the case when the log canonical divisor on the generic fiber is ample. 
We use the following notation:
\begin{notation} \label{notation:conclusion}
Let $k$ be an algebraically closed field of characteristic $p>0$. 
Let $X$ be a normal projective variety over $k$ 
and let $\Delta$ be an effective $\mathbb Q$-Weil divisor on $X$ 
such that $K_X+\Delta$ is $\mathbb Q$-Cartier. 
Let $i_{\mathrm W}$ (resp. $i_{\mathrm C}$) be the smallest positive integer 
such that $i_{\mathrm W} \Delta$ (resp. $i_{\mathrm C}(K_X+\Delta)$) is 
integral (resp. Cartier). 
Note that $i_{\mathrm W} | i_{\mathrm C}$. 
Let $U \subseteq X$ be the largest Gorenstein open subset. 
Let $f:X\to Y$ be a surjective morphism to a projective $n$-dimensional variety $Y$ over $k$. 

Let $Y_1 \subseteq Y$ be the subset consisting of points $y$ with the following properties:
\begin{itemize}
\item $(X,\Delta)$ is $F$-pure in a neighborhood of every point in $f^{-1}(y)$;
\item $(K_X+\Delta)|_{X_y}$ is ample.
\end{itemize}

Let $Y_0 \subseteq Y$ be the subset consisting of points $y$ with the following properties:
\begin{itemize}
\item $Y$ is regular at $y$; 
\item $f$ is flat at every point in $f^{-1}(y)$; 
\item $X_y$ satisfies $S_2$ and $G_1$; 
\item $\mathrm{Supp}(\Delta)$ does not contain any irreducible component of $f^{-1}(y)$; 
\item $\left(X_{\overline y},\overline{\Delta|_{U_{\overline y}}}\right)$ is $F$-pure, where $\overline{\Delta|_{U_{\overline y}}}$ is the effective $\mathbb Q$-AC divisor that is the extension of $\Delta|_{U_{\overline y}}$. 
\item $(K_X+\Delta)|_{X_y}$ is ample;
\end{itemize}
We note that 
\begin{itemize}
\item $Y_0$ and $Y_1$ are open subsets of $Y$ such that $Y_0 \subseteq Y_1 \subseteq Y$; 
\item if $f$ is \textit{not} separable, then $Y_0$ is empty, but $Y_1$ may be not empty. 
\end{itemize}
Let $H$ be an ample Cartier divisor on $Y$ with $|H|$ free. 
\end{notation}
\begin{thm} \label{thm:conclusion} \samepage
Let the notation be as in~\ref{notation:conclusion}. 
\begin{itemize}
\item[(1)] \textup{(Corollary~\ref{cor:f-ample})}
Suppose that $p \nmid i_{\mathrm W}$. 
Then there exists a positive integer $m_1$ such that 
for each $m\ge m_1$ with $i_{\mathrm C}|m$, the sheaf 
$$
f_* \mathcal O_X(m(K_{X}+\Delta)) \otimes \mathcal O_Y(lH)
$$
is globally generated over $Y_1$ for each $l \ge m(n+1)$. 
\item[(2)] \textup{(Theorem~\ref{thm:relative_psef})}
Suppose that $p\nmid i_{\mathrm C}$, that $Y$ is normal 
and that $K_Y$ is $\mathbb Q$-Cartier. 
Then there exists an integer $m_0\ge m_1$ such that 
for each $m\ge m_0$ with $i_{\mathrm C}|m$, the sheaf 
$$
f_* \mathcal O_X(m(K_{X/Y}+\Delta)) 
$$
is weakly positive over $Y_0$, and 
$$
f_* \mathcal O_X(m(K_{X/Y}+\Delta)) \otimes \mathcal O_Y(K_Y+lH)
$$
is globally generated over $Y_0$ for each $l\ge n+1$.  
\end{itemize}
\end{thm}
\begin{proof}
We only prove the first statement. 
The second is proved by the same argument. 
As shown in Example~\ref{eg:dimY+1}, 
we have an open covering $\{S_1,\ldots S_d\}$ of $Y$ 
such that for each $i = 1,\ldots, d$, 
the pair $(S_i,H)$ satisfies condition~$(\ast)_{n+1}$, 
and then so does the pair $(S_i\cap Y_1, H)$. 
Then by Corollary~\ref{cor:f-ample}, we get an integer $m_1$ such that 
$$
\mathrm{Bs}\big(
f_* \mathcal O_X(m(K_{X}+\Delta)) \otimes \mathcal O_Y(lH)
\big)
\subseteq Y \setminus \bigcup_{i} (S_i \cap Y_1) =Y\setminus Y_1, 
$$
for each $m\ge m_1$ with $i_{\mathrm C}|m$, which completes the proof. 
\end{proof}
\section{Iitaka's conjecture} \label{section:Iitaka}
Iitaka~\cite{Iit72} has proposed the following conjecture:
\begin{conj} \label{conj:Iitaka} 
Let $k$ be an algebraically closed field of characteristic zero,
let $f : X\to Y$ be a surjective morphism between 
smooth projective varieties with connected fibers, 
and let $F$ denote the geometric generic fiber of $f$. 
Then 
\begin{align}
\kappa(X) \ge \kappa(Y) + \kappa(F).
\tag{I}\label{ineq:iitaka}
\end{align}
\end{conj}
This conjecture has proved in several cases including the following: 
\begin{itemize}
\item $X$ is a surface by Ueno~\cite{Uen75}; 
\item $F$ is a curve by Viehweg~\cite{Vie77};
\item $X$ is a threefold by Viehweg~\cite{Vie80};
\item $Y$ is a curve by Kawamata~\cite{Kaw82};
\item $Y$ is of general type by Viehweg~\cite{Vie82};
\item $F$ has a good minimal model by Kawamata~\cite{Kaw85};
\item $F$ is of general type by Koll\'ar~\cite{Kol87}; 
\item $Y$ is of maximal Albanese dimension by Cao and P\u aun~\cite{CP17} and Hacon--Popa--Schnell \cite{HPS18}. 
\end{itemize}
In \cite{HPS18}, Conjecture~\ref{conj:Iitaka} is proved when $Y$ is of maximal Albanese dimension, based on \cite{CP17} that has shown the conjecture when $Y$ is an abelian variety, but according to \cite{HPS18}, the work from \cite{CP17} to \cite{HPS18} is a ``very little extra work''.
 
In this section, we study inequality (\ref{ineq:iitaka}) in positive characteristic. 
To explain several known results, we assume that $F$ is smooth. 
Then inequality (\ref{ineq:iitaka}) has proved in the following cases:
\begin{itemize}
\item $X$ is a surface by Chen and Zhang~\cite{CZ13b}; 
\item $F$ is a curve by Chen and Zhang~\cite{CZ13b};
\item $Y$ is of general type and $F$ has non-nilpotent Hasse--Witt matrix by Patakfalvi~\cite{Pat18}
\item $F$ satisfies conditions~(i) and (i\hspace{-1pt}i') in Theorem~\ref{thm:relative_intro} and $Y$ is either a curve or is of general type
by the author~\cite{Eji17};
\item $X$ is a threefold and $p \ge 7$ by the author and Zhang~\cite{EZ18} (the case when $k=\overline{\mathbb F_p}$ is due to \cite{BCZ18}, and see \cite{Zha19s} for the log version). 
\end{itemize}
In this section, we deal with an algebraic fiber space whose general fibers may have ``bad'' singularities. 
More precisely, we study inequality~(\ref{ineq:iitaka}) under some assumptions on the generic fiber, but we do not impose any condition on general fibers. 
Note that, in such a situation, counterexamples to inequality~(\ref{ineq:iitaka}) have been found by the recent work of Cascini, Koll\'ar, Zhang and the author~\cite{CEKZ20}. 
The main theorem (Theorem~\ref{thm:Iitaka}) in this section gives sufficient conditions for inequality~(\ref{ineq:iitaka}) to hold. 
To prove it, we need the following theorem. 
\begin{thm} \label{thm:c-relative}
Let $k$ be an $F$-finite field of characteristic $p>0$. 
Let $X$ be a quasi-projective normal variety, 
let $Y$ be a regular quasi-projective variety, 
and let $f:X \to Y$ be a separable surjective projective morphism. 
%
Let $\Delta$ be an effective $\mathbb Q$-Weil divisor on $X$,  
and let $i_{\mathrm W}$ be the smallest positive integer such that $i_{\mathrm W}\Delta$ is integral. 
Assume that $i_{\mathrm W}$ is not divisible by $p$. 
Suppose that 
\begin{itemize}
\item[$(\rm i)$]
$
\bigoplus_{l \ge 0} 
H^0\left( X_{\eta}, 
\mathcal O_{X_{\eta}}\left( \lfloor l(K_{X_{\eta}} +\Delta|_{X_{\eta}}) \rfloor \right)
\right)
$ 
is a finitely generated $k(\eta)$-algebra, where $\eta$ is the generic point of $Y$, and that
\item[$(\rm i\hspace{-1pt}i)$]
there exists a non-negative integer $m_0$ with $i_{\mathrm W}|m_0$ such that 
$$ 
S^0\left( X_{\eta}, \Delta|_{X_{\eta}}; \mathcal O_{X_{\eta}}(m(K_{X_\eta} +\Delta|_{X_{\eta}})) \right)
=H^0\left( X_{\eta}, \mathcal O_{X_{\eta}}(m(K_{X_{\eta}} +\Delta|_{X_{\eta}})) \right)
$$ 
for each $m \ge m_0$ with $i_{\mathrm W}|m$.
\end{itemize}
Let $H$ be a big and semi-ample Cartier divisor on $Y$ such that $(\{\eta\}, H)$ satisfies condition~$(\ast)_{a}$. 
Let $X_0$ be the largest open subset such that $K_{X_0}+\Delta|_{X_0}$ is $\mathbb Q$-Cartier, and let $f_0:X_0\to Y$ denote the induced morphism.  
Set 
$$
D:=m_0(K_X +\Delta) -f^*K_Y +{a} (m_0 -1)f^*H. 
$$ 
Then $ \mathbb B_-^{f_0^* H}(D|_{X_0}) \ne X_0$
{\rm(}see Section~\ref{section:positivity} for the definition of $\mathbb B_-^{f^*H}${\rm)}.
In particular, if $K_X+\Delta$ is $\mathbb Q$-Cartier, 
then $\kappa(X, D+\varepsilon f^*H) \ge 0$ for every $0< \varepsilon \in\mathbb Q$. 
\end{thm}
Note that we cannot remove the assumption that $f$ is separable. 
(For example, the theorem does not hold for the Frobenius morphism of a smooth projective curve of genus at least 2.)
\begin{proof}
We first prove the assertion in the case when $f$ is flat. 
For simplicity, we put $\mathcal F_m := f_* \mathcal O_X(\lfloor m(K_X+\Delta) \rfloor)$ for each $m \ge 0$. 
Fix integers $\mu$ and $n_0$ that are large and divisible enough. 
By condition~(i), for $e\gg0$, the natural morphisms 
\begin{align}
\mathcal F_{\mu q_e} \otimes \mathcal E
\twoheadrightarrow
\mathcal F_{\mu q_e} \otimes \mathcal F_{r_e}
\to
\mathcal F_{(m_0-1)p^e +1}
\tag{\ref{thm:c-relative}.1} \label{mor:A}
\end{align}
are generically surjective, where $q_e$ and $r_e$ are integers such that 
$$
(m_0-1) p^e +1 = \mu q_e + r_e 
\textup{\quad and \quad} 
n_0 \le r_e < n_0 + \mu, 
$$
and $\mathcal E:=\bigoplus_{n_0 \le r < n_0+\mu} \mathcal F_r$.
Put 
$l_e:=a(\mu (q_e +1) +n_0 -1)+1$ for each $e\in\mathbb Z_{>0}$. 
Then 
$$
l_e 
= a(\mu q_e +r_e -1 + \mu +n_0 -r_e ) +1
= a(m_0 -1)p^e +\underbrace{a(\mu +n_0 -r_e)}_{\ge0} +1, 
$$
so 
$
0 < \varepsilon_e 
:=l_e p^{-e} -a(m_0-1) 
\xrightarrow{e\to\infty} 0. 
$
Fix some $\nu\in\mathbb Z_{>0}$ and put $L := m_0(K_X+\Delta)$.  
We show that 
$$
\mathbb B(p^e(L|_{X_0}-f_0^*K_Y) +(l_e+\nu) f_0^*H) \ne X_0
$$ 
for each $e\gg0$.  If this holds, then 
\begin{align*}
\mathbb B_-^{f_0^*H}(D|_{X_0}) 
= &\mathbb B_-^{f_0^*H}(L|_{X_0} -f_0^*K_Y +a(m_0-1)f_0^*H) 
\\ = & \bigcup_{e\gg0} \mathbb B\bigg(
L|_{X_0}-f_0^*K_Y +\bigg( 
\underbrace{a(m_0 -1) +\varepsilon_e}_{\textup{\hspace{20pt}$=\frac{l_e}{p^e}$}} +\frac{\nu}{p^e} \bigg)f_0^*H
\bigg)
\\ = & \bigcup_{e\gg0} \mathbb B\left(
p^e(L|_{X_0}-f_0^*K_Y) +\left( l_e +\nu \right)f_0^*H
\right)
\ne X_0,
\end{align*}
which prove the assertion. 
 
Since $f$ is flat, we may apply the argument in Section~\ref{section:trace}.  
Put 
$
M:=L_{Y^e} +(1-p^e){f_{Y^e}}^*K_Y . 
$ 
By (i\hspace{-1pt}i), we see that $f_* \phi^{(e)}_{(X,\Delta)}(L)$ is non-zero, 
so 
\begin{align*}
\tag{\ref{thm:c-relative}.2} \label{mor:B}
\underbrace{
{f^{(e)}}_*\mathcal O_{X^e}(p^e L +(1-p^e) (K_X +\Delta) )
}_{\hspace{55pt} =\mathcal F_{(m_0 -1)p^e +1}}
 \xrightarrow{{f_{Y^e}}_* \phi^{(e)}_{(X/Y,\Delta)}(M)}
{f_{Y^e}}_* \mathcal O_{X_{Y^e}}(M)
\end{align*}
is also non-zero as explained in Section~\ref{section:trace}.
Applying 
$
{f_{Y^e}}_*\left( ? \otimes \mathcal O_{X_{Y^e}}(M) \right)
$
to the natural morphism $\mathcal O_{X_{Y^e}} \hookrightarrow {F_{X/Y}^{(e)}}_* \mathcal O_{X^e}$, 
we get 
\begin{align}
	\tag{\ref{thm:c-relative}.3}\label{mor:C}
{f_{Y^e}}_* \mathcal O_{X_{Y^e}}(M) 
\hookrightarrow
{f^{(e)}}_* \mathcal O_{X^e}\left({F_{X/Y}^{(e)}}^*M \right). 
\end{align}
Combining morphisms (\ref{mor:A}), (\ref{mor:B}) and (\ref{mor:C}), 
we get the morphisms 
\begin{align}
\tag{\ref{thm:c-relative}.4}\label{mor:D}
\mathcal F_{\mu q_e} \otimes \mathcal E
\to 
f_* \mathcal O_{X}\left({F_{X/Y}^{(e)}}^*M \right) 
\cong 
{f}_* \mathcal O_{X}\left(p^e L +(1-p^e){f}^*K_Y \right) 
\end{align}
whose composite is non-zero. 
Take $\nu_1 \in \mathbb Z_{>0}$ so that $K_Y \le \nu_1 H$. 
Then, we may replace the right-hand side of (\ref{mor:D}) with
$
f_*\mathcal O_X(p^e(L-f^*K_Y) +\nu_1f^*H). 
$
Pick $\nu_2\in \mathbb Z_{>0}$ so that $\mathcal E(\nu_2 H)$ is 
generically globally generated, and put $\nu:=\nu_1+\nu_2$. 
By (\ref{mor:D}), we have the non-zero morphism 
\begin{align}
\tag{\ref{thm:c-relative}.5}\label{mor:E}
\mathcal F_{\mu q_e} \otimes \mathcal E (\nu_2 H)
\to 
f_* \mathcal O_{X}\left(p^e (L -f^*K_Y) +\nu f^*H \right) 
=
f_*\mathcal O_X(D_e), 
\end{align}
where $D_e := p^e (L-f^*K_Y) +\nu f^*H$.
Note that $\nu$ is independent of $e$. 
By the choice of $\nu_2$, there is a non-zero morphism 
$\mathcal F_{\mu q_e} \to f_*\mathcal O_X(D_e)$ 
and its adjoint $\varphi:f^*\mathcal F_{\mu q_e} \to \mathcal O_X(D_e)$. 
Put $C:=\mathrm{Supp}(\mathrm{Coker}\,\varphi)$. 
Then $C\ne X$ and 
$$
\mathbb B (\mathcal O_X(D_e) +l_ef^*H) 
\overset{\textup{by $\varphi$}}{\subseteq}
\mathbb B (f^* \mathcal F_{\mu q_e} +l_ef^*H) 
\cup
C
\subseteq 
f^{-1}\big(\mathbb B(\mathcal F_{\mu q_e} +l_eH)\big) \cup C . 
$$
Since $l_e > a\mu q_e$ by definition, we have 
\begin{align*}
\mathbb B \left(\mathcal F_{\mu q_e} +l_eH\right)
\subseteq 
\mathbb B_-^H \left(\mathcal F_{\mu q_e} + a\mu q_e H \right) 
\overset{\textup{Cor~\ref{cor:AD}}}{\subseteq}
\mathbb B_- \left(\mathcal F_{\mu q_e} + a\mu q_e H \right) 
\cup \mathbb B_+(H)
\overset{\textup{Thm~\ref{thm:main}}}{\ne} 
Y, 
\end{align*}
so $\mathbb B(\mathcal O_X(D_e) +l_ef^*H) \ne X$,  
and hence 
$
\mathbb B \left(D_e|_{X_0} +l_e f_0^*H \right) \ne X_0, 
$
which is our claim. 

Next, we show the assertion in the case when $f$ is not necessarily flat. 
Let $X' \to Y'$ be the flattening of $f$. 
Let $Y''$ be the normalization of $Y'$, 
let $X''$ be the normalization of the main component of $X'\times_{Y'}Y''$,
and let $f'':X'' \to Y''$ be the induced morphism. 
Let $V$ be the regular open subset of $Y''$ such that 
$g:=f''|_U: U \to V$ is flat, where $U:={f''}^{-1}(V)$. 
Now we have the following commutative diagram:
\begin{align*}
\xymatrix{ 
X \ar[d]_{f} & U \ar[l]_{\rho} \ar[d]^{g} \\
Y & V \ar[l]_{\sigma} 
}
\end{align*}
Note that $(\{\eta_V\}, \sigma^*H)$ satisfies $(\ast)_{a}$ as shown in Lemma~\ref{lem:epsilon},
where $\eta_V$ is the generic point of $V$. 
Put $\Delta' := \rho^{-1}_* \Delta$. 
Since $\sigma^* K_Y \le K_V$, 
by the above argument,
we see that for each $t \in \mathbb Q_{>0}$ there is a $\mathbb Q$-Weil divisor $E_t \ge 0$ on $U$
such that 
$$
E_t \sim_{\mathbb Q} 
m_0(K_U +\Delta') -g^*\sigma^*K_Y +({a}(m_0-1) +t) g^*\sigma^*H. 
$$
Applying $\rho_*$, we get 
$$
0 \le \rho_* E_t \sim_{\mathbb Q} 
m_0(K_X +\Delta) -f^*K_Y +({a}(m_0-1) +t) f^*H, 
$$
which completes the proof. 
\end{proof}
\begin{thm} \label{thm:Iitaka}
Let $k$ be an algebraically closed field of characteristic $p>0$. 
Let $X$ be a normal projective variety 
and let $\Delta$ be an effective $\mathbb Q$-Weil divisor on $X$
such that $i_{\mathrm C}(K_X +\Delta)$ is Cartier for an integer $i_{\mathrm C}>0$ not divisible by $p$. 
Let $Y$ be a smooth projective variety of maximal Albanese dimension, 
and let $f:X\to Y$ be a separable surjective morphism. 
Suppose that the following conditions hold:
\begin{itemize}
\item[$(\rm i)$]
$
\bigoplus_{l \ge 0} 
H^0\left( X_{\eta}, 
\mathcal O_{X_{\eta}}\left( \lfloor l(K_{X_{\eta}} +\Delta|_{X_{\eta}}) \rfloor \right)
\right)
$ 
is a finitely generated $k(\eta)$-algebra, where $\eta$ is the generic point of $Y$; 
\item[$(\rm i\hspace{-1pt}i)$]
there exists an integer $m_0\ge 0$ with $i_{\mathrm C}|m_0$ such that 
$$ 
S^0\left( X_{\eta}, \Delta|_{X_{\eta}}; \mathcal O_{X_{\eta}}(m(K_{X_\eta} +\Delta|_{X_{\eta}})) \right)
=H^0\left( X_{\eta}, 
\mathcal O_{X_{\eta}}(m(K_{X_{\eta}} +\Delta|_{X_{\eta}})) 
\right)
$$ 
for each $m \ge m_0$ with $i_{\mathrm C}|m$; 
\item[$(\rm i\hspace{-1pt}i\hspace{-1pt}i)$]
either $Y$ is a curve or is of general type. 
\end{itemize}
Then 
$$
\kappa\left(X, K_X +\Delta\right) 
\ge 
\kappa(Y) + \kappa\left(X_{\eta}, K_{X_{\eta}} +\Delta|_{X_{\eta}}\right). 
$$
\end{thm}
\begin{proof}
We may assume that $\kappa(Y)\ge 0$ and $\kappa(X_{\eta},K_{X_{\eta}}+\Delta|_{X_{\eta}}) \ge 0$.
Since $Y$ is of maximal Albanese dimension, 
$(\{\eta\},H)$ satisfies condition~$(\ast)_0$ for some ample Cartier divisor $H$ on $Y$, 
as shown in Example~\ref{eg:mAd}. 

First we deal with the case when $Y$ is of general type. 
The proof is similar to that of \cite[Theorem~7.2]{Eji17}, 
which is based on the proof of \cite[Theorem~1.7]{Pat18}. 
Let $H$ be an ample Cartier divisor on $Y$. 
Put $D:=m_0(K_X+\Delta) -f^*K_Y$, 
\begin{align*}
S & := \left\{\varepsilon \in \mathbb Q \middle| 
\kappa(X, D -\varepsilon f^*H) \ge 0\right\} \textup{~and~}
\\ S' & := \left\{\varepsilon \in \mathbb Q \middle| 
\kappa(X, D -\varepsilon f^*H) \ge \kappa(X_{\eta}, D_{\eta}) +\dim Y\right\}. 
\end{align*}
Then $S'\subseteq S$. 
We show $S'\ne\emptyset$. 
By assumption~(i), there is a $\mu \in \mathbb Z_{>0}$ such that 
the $k(\eta)$-algebra 
$
\bigoplus_{l\ge 0} H^0\left(X_\eta, \mathcal O_{X_\eta}\left(l\mu D_\eta\right)\right)
$
is generated by $H^0\left(X_\eta, \mu D_\eta\right)$. 
Using the projection formula, we can find a $\nu\in\mathbb Z_{>0}$ such that 
$
f_*\mathcal O_X(\mu D +\nu f^*H)
$
is generated by its global sections. 
By the choice of $\mu$, the natural morphism 
$$
\bigotimes^n f_*\mathcal O_X(\mu D+\nu f^*H)
\to f_*\mathcal O_X(n (\mu D + \nu f^*H))
$$
is generically surjective for each $n\in\mathbb Z_{>0}$, so 
$
f_*\mathcal O_X(n (\mu D + \nu f^*H))
$
is generically generated by its global sections. 
From this, we get the injective morphism 
$$
\bigoplus^{\mathrm{rank}\,f_*\mathcal O_X(n\mu D)}
\mathcal O_Y(nH) \to f_*\mathcal O_X\big(n(\mu D+(\nu+1)f^*H)\big). 
$$
Since $\mathrm{rank}\, f_*\mathcal O_X(n\mu D) = \dim H^0(X_\eta, n\mu D_\eta)$, we obtain that 
$$
\dim H^0(Y,nH) \times \dim H^0(X_\eta, n\mu D_\eta) \le \dim H^0\big(X, n(\mu D+(\nu +1)f^* H)\big) 
$$
for each $n\in\mathbb Z_{>0}$, so 
$$
\dim Y+ \kappa(X_\eta, D_\eta) = \kappa(Y,H) +\kappa(X_\eta, \mu D_\eta)
\le \kappa(X, \mu D+(\nu +1)f^*H),  
$$
and hence $\varepsilon_0 := -\mu^{-1}(\nu +1) \in S'$. 

We prove $\sup S' = \sup S$. 
The inequality $\le$ is obvious. We show $\ge$. Take $\varepsilon \in S$. 
Then $D-\varepsilon f^*H$ is $\mathbb Q$-linearly equivalent to an effective $\mathbb Q$-Cartier divisor, so for every $\delta\in\mathbb Q_{>0}$ we have 
\begin{align*}
\kappa(X, (1+\delta)D -(\varepsilon +\delta \varepsilon_0)f^*H)
\ge \kappa(X, \delta (D -\varepsilon_0 f^*H))
\ge \kappa(X_\eta, D_\eta) +\dim Y. 
\end{align*}
Hence, we get
$
S' \ni (1+\delta)^{-1}(\varepsilon +\delta\varepsilon_0)
\xrightarrow{\delta \to 0} \varepsilon, 
$
which means that $\sup S' \ge \varepsilon$, and so $\sup S' \ge \sup S$. 

We prove the assertion. By Theorem~\ref{thm:c-relative}, 
we see that $\mathbb Q_{<0} \subseteq S$, so $0\le \sup S = \sup S'$. 
Therefore, we find $\varepsilon \in S'$ such that 
$K_Y+\varepsilon H$ is $\mathbb Q$-linearly equivalent to an effective $\mathbb Q$-divisor. Note that $K_Y$ is big. 
Hence, we get 
\begin{align*}
\kappa(X,K_X+\Delta) 
= & \kappa(X,D+f^*K_Y)
\\ \ge & \kappa(X, D+f^*K_Y -f^*(K_Y +\varepsilon H))
\\ = & \kappa(X, D -\varepsilon f^*H) 
\\ \ge & \kappa(X_\eta, D_\eta) +\dim Y
\\ = & \kappa(X_\eta, K_{X_\eta} +\Delta_\eta) +\kappa(Y). 
\end{align*}

Next, we consider the case when $Y$ is an elliptic curve. 
Let $\mu$ be an integer that is large and divisible enough. 
Set $\mathcal G := f_*\mathcal O_X(\mu(K_X +\Delta))$. 
Then $\mathcal G$ is a nef vector bundle by Theorem~\ref{thm:main}.
For a nef vector bundle $\mathcal V$ on $Y$, 
let $\mathbb L(\mathcal V)$ denote the subset of $\mathrm{Pic}^0(Y)$ 
consisting of line bundles $\mathcal L$ on $Y$ 
that can be obtained as a quotient bundle of $\mathcal V$. 
Let $G$ be the subgroup of $\mathrm{Pic}^0(Y)$ generated by $\mathbb L(\mathcal G)$. 
We prove that $G$ is a finite group. 
If this holds, then each $\mathcal L \in \mathbb L(\mathcal G)$ is a torsion line bundle, 
which means that there is a finite morphism $\pi:Y'\to Y$ 
from an elliptic curve $Y'$ such that $\pi^*\mathcal G$ is generated by its global sections,
and hence we can prove the assertion by applying the same argument as that in the proof of \cite[Theorem~7.6]{Eji17}. 

We here use the classification of vector bundles on an elliptic curve \cite{Ati57, Oda71}. See \cite[Theorem~7.3]{Eji17} for a summary. 
Considering the decomposition of $\mathcal G$ into indecomposable vector bundles, we see that 
\begin{itemize}
\item $\mathcal G$ is isomorphic to the direct sum of an ample vector bundle $\mathcal G^+$ and 
a nef vector bundle $\mathcal E$ of degree 0, and 
\item 
$\mathcal E \cong \bigoplus_{1\le j \le \nu}  \mathcal E_{r_j,0} \otimes \mathcal L_j $, 
where $\mathcal L_j \in \mathrm{Pic}^0(Y)$ 
and $\mathcal E_{r_j,0}$ is an indecomposable vector bundle 
of rank $r_j$ and degree 0 having a non-zero global section. 
\end{itemize}
Note that, to show the first statement, we used the fact that indecomposable vector bundles of positive degree are ample, which is obtained by, for example, combining \cite[Theorem~7.3~(3)]{Eji17} and \cite[Theorem~2.16]{Oda71}. 
Since $\mathcal E_{r,0}$ is an extension of $\mathcal E_{r-1,0}$ by $\mathcal O_Y$, 
we get a filtration
$$
0=\mathcal G_0 \subset \mathcal G_1 \subset \cdots \subset \mathcal G_{\rho+1} =\mathcal G 
\quad \textup{($\rho := \mathrm{rank}(\mathcal E)$)}
$$ 
of $\mathcal G$ such that 
$
\{\mathcal G_1/\mathcal G_0, \ldots, \mathcal G_\rho/\mathcal G_{\rho-1} \} =\mathbb L(\mathcal E) (=\mathbb L(\mathcal G))
$
and $\mathcal G_{\rho+1}/\mathcal G_{\rho} \cong \mathcal G^+$. 
 
Put $\mathbb L(\mathcal E) =\{ \mathcal L_1,\ldots,\mathcal L_\lambda\}. $
For each $m \ge 0$, set 
$$
G(m):=\left\{ \mathcal L_1^{m_1} \otimes \cdots \otimes \mathcal L_\lambda^{m_\lambda}
\middle|  |m_1| + \cdots + |m_\lambda| \le m
\right\}.
$$ 
Then for each $\mathcal N_1 \in G(n_1)$ and $\mathcal N_2 \in G(n_2)$, 
we have $\mathcal N_1 \otimes \mathcal N_2 \in G(n_1 +n_2)$. 
Set 
$$
\mathcal F := \bigoplus_{\textup{$n_0 \le r < n_0 +\mu$ and $i_C|r$}} f_* \mathcal O_X(r(K_X +\Delta)) 
$$ 
for some $n_0 \gg 0$. 
Then $\mathcal F$ is a nef vector bundle by Theorem~\ref{thm:main}. 
Since $\mathbb L(\mathcal F)$ is a finite set, 
there is $\nu \in \mathbb Z_{>0}$ such that 
$\mathbb L(\mathcal F) \cap G \subseteq G(\nu).$
Take $e \gg 0$ so that $p^e -q_e >\nu $, 
where $q_e$ is an integer such that $(\mu-1)p^e +1 = \mu q_e +r_e$ 
for an integer $r_e$ with $n_0 \le r_e < n_0 +\mu$. 
We have the generically surjective morphisms
$$
{F_Y^e}_* \left( 
S^{q_e}(\mathcal G) \otimes \mathcal F
\right)
\twoheadrightarrow
{F_Y^e}_* \left( 
S^{q_e}(\mathcal G) \otimes f_* \mathcal O_X(r_e(K_X +\Delta))
\right)
\to
\mathcal G
$$
by the same argument as that in the proof of Theorem~\ref{thm:general}. 
For each $\mathcal L_j \in \mathbb L(\mathcal E)$, 
we have $\mathcal G \twoheadrightarrow \mathcal L_j$, 
which induces the non-zero morphism
$$
{F_Y^e}_* \left( 
S^{q_e}(\mathcal G) \otimes \mathcal F
\right)
\to
\mathcal L_j. 
$$
Since $\omega_Y\cong\mathcal O_Y$, 
by \cite[I\hspace{-1pt}I\hspace{-1pt}I, Proposition~6.9~a)]{Har66} we get 
$$
{F_Y^e}^!\mathcal L_j 
\cong \left( {F_Y^e}^!\mathcal O_Y \right) \otimes {F_Y^e}^*\mathcal L_j
\cong \omega_Y^{1-p^e} \otimes \mathcal L_j^{p^e}
\cong \mathcal L_j^{p^e}, 
$$
so it follows from Grothendieck duality that  
$$
0\ne\mathrm{Hom}({F_Y^e}_*(S^{q_e}(\mathcal G)\otimes\mathcal F), \mathcal L_j)
\cong\mathrm{Hom}\left(S^{q_e}(\mathcal G)\otimes\mathcal F, \mathcal L_j^{p^e}\right),
$$
and hence we get the non-zero morphism
$$
S^{q_e}(\mathcal G) \otimes \mathcal F
\to \mathcal L_j^{p^e}.
$$
This is surjective, since the source is a nef vector bundle and $\deg \mathcal L_j^{p^e}=0$. 
Considering the above filtration, 
we find $m_1,\ldots,m_\lambda \in \mathbb Z_{\ge 0}$ 
with $\sum_{j=1}^{\lambda} m_j =q_e$ 
and $\mathcal M \in \mathbb L(\mathcal F)$ 
such that 
$$
\mathcal L_1^{m_1} \otimes \cdots \otimes \mathcal L_\lambda^{m_\lambda} 
\otimes \mathcal M
\cong 
\mathcal L_j^{p^e},
$$
which means that $\mathcal L_j^{p^e} \in G(q_e +\nu)$. 
 
Set $N := \lambda (p^e-1)$.  
Suppose that there is $ \mathcal L \in G \setminus G(N)$. 
Let $M$ be the minimal integer such that $\mathcal L \in G(M)$. 
Then we find $m_1,\ldots,m_\lambda$ with $\sum_{j=1}^\lambda |m_j| =M$ such that 
$ 
\mathcal L \cong \mathcal L_1^{m_1} \otimes \cdots \otimes \mathcal L_\lambda^{m_\lambda}. 
$
We see from $M>N$ that $|m_j| \ge p^e$ for some $j$, so 
$
\mathcal L_j^{|m_j|} \cong \mathcal L_j^{|m_j|-p^e} \otimes \mathcal L_j^{p^e}, 
$
which means that $\mathcal L_j^{m_j} \in G(|m_j|-p^e +q_e +\nu)$, 
and hence 
$$
\mathcal L \in G(|m_j|-p^e +q_e +\nu +M -|m_j|) =G(M -(p^e -q_e -\nu) ),
$$
but this contradicts the choice of $M$, since $p^e -q_e -\nu>0$. 
Hence we conclude that $G=G(N)$, which proves our claim. 
\end{proof}
\bibliographystyle{abbrv}
\bibliography{ref}

\end{document}